\documentclass[10pt,final]{amsart}

\usepackage[english]{babel}
\usepackage{latexsym}
\usepackage{amssymb}
\usepackage{amscd}
\usepackage[latin1]{inputenc}
\usepackage[mathscr]{eucal}

\newtheorem{teor}{Theorem}[section]
\newtheorem{prop}[teor]{Proposition}
\newtheorem{cor}[teor]{Corollary}
\newtheorem{lema}[teor]{Lemma}

\theoremstyle{definition}
\newtheorem{defin}[teor]{Definition}

\theoremstyle{remark}
\newtheorem{remark}[teor]{Remarks}

\newcommand{\fin}{\hspace*{\fill} $\Box$}

\newcommand{\R}{\mathbb{R}}
\newcommand{\N}{\mathbb{N}}

\newcommand{\To}{\longrightarrow}

\newcommand{\U}{\mathscr{U}}

\def\Ext{\operatorname{Ext}}

\def\PO{\operatorname{PO}}
\def\PB{\operatorname{PB}}

\def\fin{\operatorname{fin}}

\newcommand{\aproof}{\begin{proof}}
\newcommand{\zproof}{\end{proof}}

\def\XiU{[X_i]_{\mathscr U}}
\def\Xii{(X_i)_{i\in I}}

\def\dens{\operatorname{dens}}
\def\dom{\operatorname{dom}}

\def\e{\varepsilon}

\def\N{\mathbb{N}}

\def\span{\operatorname{span}}
\def\R{\mathbb R}
\def\PB{\operatorname{PB}}
\def\ran{\operatorname{ran}}

\begin{document}

\title{On separably injective Banach spaces}

\author[A. Avil\'es, F. Cabello S\'anchez, J. M. F. Castillo, M. Gonz\'alez and Y.
Moreno]{Antonio Avil\'es, F\'elix Cabello S\'anchez, Jes\'us M. F.
Castillo, Manuel Gonz\'alez and Yolanda Moreno}

%\author{Antonio Avil\'es}
\address{Departamento de Matem\'aticas, Universidad de Murcia, 30100 Espinardo,
Murcia, Spain} \email{avileslo@um.es}

%\author{F\'elix Cabello S\'anchez}
\address{Escuela de Ingenieros T\'ecnicos Industriales, Universidad de Extremadura, Avenida de Elvas s/n, 06071 Badajoz, Spain}
             \email{fcabello@unex.es}

%\author{Jes\'us M. F. Castillo}
\address{Departamento de Matem\'aticas, Universidad de Extremadura, Avenida de Elvas s/n, 06071 Badajoz, Spain}
             \email{castillo@unex.es}

%\author{Manuel Gonz\'alez}
\address{Departamento de Matem\'aticas, Universidad de Cantabria, Avenida los Castros s/n, 39071 Santander, Spain}
             \email{manuel.gonzalez@unican.es}

%\author{Yolanda Moreno.}
\address{Escuela Polit\'ecnica, Universidad de Extremadura, Avenida de la Universidad s/n, 10071 C\'aceres, Spain}
             \email{ymoreno@unex.es}
%    Remove any unused author tags.

%    author one information

%Departamento de Matem\'aticas, Universidad de Extremadura}

\thanks{2010 Class. subject: 46A22, 46B04, 46B08,
46B26.}
\thanks{The first author was supported by MEC and FEDER (Project
MTM2008-05396), Fundaci\'{o}n S\'{e}neca (Project 08848/PI/08) and
Ramon y Cajal contract (RYC-2008-02051). The research of the other
four authors has been supported in part by project MTM2010-20190}

\maketitle

\begin{abstract}In this paper we deal with two weaker forms of injectivity which
turn out to have a rich structure behind: separable injectivity
and universal separable injectivity.  We show several structural
and stability properties of these classes of Banach spaces. We
provide natural examples of (universally) separably injective
spaces, including $\mathcal L_\infty$ ultraproducts built over
countably incomplete ultrafilters, in spite of the fact that these
ultraproducts are never injective. We obtain two fundamental
characterizations of universally separably injective spaces: a) A
Banach space $E$ is universally separably injective if and only if
every separable subspace is contained in a copy of $\ell_\infty$
inside $E$. b) A Banach space $E$ is universally separably
injective if and only if for every separable space $S$ one has
$\Ext(\ell_\infty/S, E)=0$. The final Section of the paper focuses
on special properties of $1$-separably injective spaces.
Lindenstrauss\ obtained in the middle sixties a result that can be
understood as a proof that, under the continuum hypothesis,
$1$-separably injective spaces are $1$-universally separably
injective; he left open the question in {\sf ZFC}. We construct a
consistent example of a Banach space of type $C(K)$ which is
$1$-separably injective but not $1$-universally separably
injective.
\end{abstract}

\section{Introduction}\label{ch:preliminaires}

A Banach space $E$ is said to be injective if for every Banach
space $X$ and every subspace $Y$ of $X$, each operator $t \colon Y
\to E$ admits an extension $T\colon X \to E$. And $E$ is said to
be $\lambda$-injective if, besides, $T$ can be chosen so that
 $\|T\|\leq \lambda \|t\|$. The space $\ell_infty$ is the best known example of 1-injective
space. The two basic facts about injective spaces are that
$1$-injective spaces are isometric to $C(K)$-spaces with $K$
extremely disconnected and that is not known if all injective
spaces are isomorphic to
$1$-injective spaces.\\

In this paper we deal with two weaker forms of injectivity which
turn out to have a rich structure behind: separable injectivity
and universal separable injectivity. A Banach space $E$ is said to
be separably injective if it satisfies the extension property in
the definition of injective spaces under the restriction that $X$
is separable; and it is said to be universally separably injective
if it satisfies the extension property when $Y$ is separable.
Obviously, injective spaces are universally separably injective
and these, in turn, are separably injective; the converse
implications fail.\\

The basic structural and stability properties of these classes are
studied in Section 3: we will show that infinite-dimensional
separably injective spaces are $\mathcal{L}_\infty$-spaces,
contain $c_0$ and have Pe\l czy\'nski's property $(V)$.
Universally separably injective spaces, moreover, are Grothendieck
spaces, contain $\ell_\infty$ and enjoy Rosenthal's property
$(V)$. In Section 4 we provide natural examples of (universally)
separably injective spaces, including the remarkable fact that
ultraproducts built over countably incomplete ultrafilters are
universally separably injective as long as they are $\mathcal
L_\infty$-spaces, in spite of the fact that they are never
injective. In Section 5 we obtain two fundamental
characterizations of universally separably injective spaces: a) A
Banach space $E$ is universally separably injective if and only if
every separable subspace is contained in a copy of $\ell_\infty$
inside $E$. b) A Banach space $E$ is universally separably
injective if and only if for every separable space $S$ one has
$\Ext(\ell_\infty/S, E)=0$; i.e., $E$ is complemented in any
superspace $Z$ such that $Z/E = \ell_\infty/S$. Characterization
(a) allows to prove that universal separable injectivity is a
$3$-space property, which provides many new examples of spaces
with this property. Characterization (b)leads to the result
$\Ext(\ell_\infty/c_0,\ell_\infty/c_0)=0$, which provides a new
unexpected solution for equation $\Ext(X,X)=0$. The final Section
6 focuses on special properties of $1$-separably injective spaces.
This is the point in which set theory axioms enter the game.
Indeed, Lindenstrauss\ obtained in the middle sixties what can be
understood as a proof that, under the continuum hypothesis,
$1$-separably injective spaces are $1$-universally separably
injective; he left open the question in {\sf ZFC}. We construct a
consistent example of a Banach space of type $C(K)$ which is
$1$-separably injective but not $1$-universally separably
injective.

\section{Background}

Our notation is fairly standard, as in \cite{lindtzaf}. Given a
set $\Gamma$ we denote by $\ell_\infty(\Gamma)$ the space of all
bounded scalar functions on $\Gamma$, endowed with the sup norm
and $c_0(\Gamma)$ is the closed subspace spanned by the
characteristic functions of the singletons of $\Gamma$. A Banach
space $X$ is said to be an $\mathcal{L}_{p,\lambda}$-space (with
$1\leq p \leq \infty$ and $\lambda \geq 1$) if every finite
dimensional subspace $F$ of $X$ is contained in another finite
dimensional subspace of $X$ whose Banach-Mazur distance to the
corresponding $\ell_p^n$ is at most $\lambda$. A space $X$ is said
to be a $\mathcal{L}_p$-space if it is a
$\mathcal{L}_{p,\lambda}$-space for some $\lambda \geq 1$; we will
say that it is a $\mathcal{L}_{p,\lambda+}$-space when it is a
$\mathcal{L}_{p,\lambda'}$-space for all $\lambda'>\lambda$. We
are especially interested in $\mathcal L_\infty$ spaces. A
Lindenstrauss\ space is one whose dual is isometric to $L_1(\mu)$
for some measure $\mu$. Lindenstrauss\ spaces and
$\mathcal{L}_{\infty,1+}$-spaces are identical classes. Throughout
the paper, {\sf ZFC} denotes the usual setting of set theory with
the Axiom of Choice, while {\sf CH} denotes the continuum
hpothesis.

\subsection{The push-out and pull-back constructions} The push-out construction appears naturally when one
considers a couple of operators defined on the same space, in
particular in any extension problem. Let us explain why. Given
operators $\alpha:Y\to A$ and $\beta:Y\to B$, the associated
push-out diagram is
\begin{equation}\label{po-dia}
\begin{CD}
Y@>\alpha>> A\\
@V \beta VV @VV \beta' V\\
B @> \alpha' >> \PO
\end{CD}
\end{equation}
Here, the push-out space $\PO=\PO(\alpha,\beta)$ is quotient of
the direct sum $A\oplus_1 B$ (with the sum norm) by the closure of
the subspace $\Delta=\{(\alpha y,-\beta y): y\in Y\}$. The map
$\alpha'$ is given by the inclusion of $B$ into $A\oplus_1 B$
followed by the natural quotient map $A\oplus_1 B\to (A\oplus_1
B)/\overline\Delta$, so that $\alpha'(b)=(0,b)+\overline\Delta$
and, analogously, $\beta'(a)=(a,0)+\overline\Delta$.

The diagram (\ref{po-dia}) is commutative:
$\beta'\alpha=\alpha'\beta$. Moreover, it is `minimal' in the
sense of having the following universal property: if $\beta'':A\to
C$ and $\alpha'':B\to C$ are operators such that
$\beta''\alpha=\alpha''\beta$, then there is a unique operator
$\gamma:\PO\to C$ such that $\alpha''=\gamma\alpha'$ and
$\beta''=\gamma\beta'$. Clearly, $\gamma((a, b) +
\overline{\Delta}) = \beta''(a)+\alpha''(b)$ and one has
$\|\gamma\|\leq \max \{\|\alpha''\|, \|\beta''\|\}$.

Regarding the behaviour of the maps in Diagram~\ref{po-dia}, apart
from the obvious fact that both $\alpha'$ and $\beta'$ are
contractive, we have:

\begin{lema}\label{isom}$\;$
\begin{itemize}
\item[(a)] $\max \|\alpha'\|\, \|\beta'\| \leq 1$. \item[(b)] If
$\alpha$ is an isomorphic embedding, then $\Delta$ is closed.
\item[(c)] If $\alpha$ is an isometric embedding and
$\|\beta\|\leq 1$ then $\alpha'$ is an isometric embedding.
\item[(d)] If $\alpha$ is an isomorphic embedding then $\alpha'$
is an isomorphic embedding. \item[(e)] If $\|\beta\|\leq 1$ and
$\alpha$ is an isomorphism then $\alpha'$ is an isomorphism and
$$\|(\alpha')^{-1}\|\leq \max \{1, \|\alpha\|\}.$$
\end{itemize}
\end{lema}
\begin{proof} (a) and (b) are clear. (c) If $\|\beta\|\leq 1$,
$$
\|\alpha'(b)\| = \|(0,b)+ \Delta\| = \inf_{y\in Y}  \|\alpha y\|+
\|b - \beta y\|  \geq  \inf_{y} \|\beta y\|+ \|b - \beta y\| \geq
\| b\|,
$$
as required. (d) is clear after (c). (e) To prove the assertion
about $(\alpha')^{-1}$, notice that for all $a\in A$ and $b\in B$
one has $(a,b) + \delta = (0, b+\beta y) + \delta$ for $y\in Y$
such that $\alpha y=a$. Therefore, for all $y'\in Y$ one has
\begin{eqnarray*} \|b +
\beta y \| &\leq & \|b + \beta y + \beta y'\| + \|\beta y'\| \\
&\leq& |b + \beta y + \beta y'\| + \|y'\| \\ &\leq& \|b + \beta y
+ \beta y'\| + \|\alpha^{-1}\| \|\alpha y'\|\end{eqnarray*} from
where the assertion follows.
\end{proof}

The pull-back construction is the dual of that of push-out in the
sense of categories, that is, ``reversing arrows''. Indeed, let
$\alpha:A\to Z$ and $\beta:B\to Z$ be operators acting between
Banach spaces. The associated pull-back diagram is
\begin{equation}\label{pb-dia}
\begin{CD}
\PB@>{'}\!\beta>> A\\
@V {'}\!\alpha VV @VV \alpha V\\
B @> \beta >> Z
\end{CD}
\end{equation}
Here, the pull-back space is $\PB=\PB(\alpha,\beta)=\{(a,b)\in
A\times B: \alpha(a)=\beta(b) \}$, where $A\times B$ carries the
sup norm. The arrows after primes are the restriction of the
projections onto the corresponding factor. Needless to say
(\ref{pb-dia}) is minimally commutative in the sense that if the
operators ${''}\!\beta: C\to A$ and ${''}\!\alpha: C\to B$ satisfy
$\alpha\circ {''}\! \beta=\beta\circ {''}\!\alpha$, then there is
a unique operator $\gamma:C\to\PB$ such that
${''}\!\beta={'}\!\beta \gamma$ and ${''}\!\beta={'}\!\beta
\gamma$. Clearly, $\gamma(c)=({''}\!\beta(c), {''}\!\alpha(c))$
and  $\|\gamma\|\leq \max \{\|{''}\!\alpha\|, \|{''}\!\beta\|\}$.

Quite clearly ${'}\!\alpha$ is onto if $\alpha$ is.\\

\subsection{Exact sequences} A short exact sequence of Banach spaces is a diagram
\begin{equation}\label{SEX}
0 \To Y \stackrel{\imath}\To X \stackrel{\pi}\To Z \To 0
\end{equation}
where $Y$, $X$ and $Z$ are Banach spaces and the arrows are
operators in such a way that the kernel of each arrow coincides
with the image of the preceding one. By the open mapping theorem
$\imath$ embeds $Y$ as a closed subspace of $X$ and $Z$ is
isomorphic to the quotient $X/\imath(Y)$.

We say that $0 \to Y \to X_1 \to Z\to0$ is equivalent to
(\ref{SEX}) if there exists an operator $t: X \to X_1$ making
commutative the diagram
\begin{equation}\label{equiv}
\begin{CD}
0  @>>>  Y @>>> X @>>> Z @>>>0 \\
  &  &   @|  @VVtV  @|
   & \\
0  @>>>  Y @>>> X_1 @>>> Z @>>>0. \\
\end{CD}
\end{equation}
This is a true equivalence relation in view of the classical
`three-lemma' asserting that in a commutative diagram of vector
spaces and linear maps $$
\begin{CD}
0  @>>>  Y @>>> X @>>> Z @>>>0 \\
  &     &@V{u}VV       @V{t}VV  @VV{v}V \\
0  @>>>  Y_1 @>>> X_1 @>>> Z_1 @>>>0.
\end{CD}
$$
having exact rows, if $u$ and $v$ are surjective (resp.,
injective) then so is $t$. Hence, the operator $t$ in Diagram
\ref{equiv} is a bijection and so it is a linear homeomorphism,
according to the open mapping theorem.

The sequence (\ref{SEX}) is said to be trivial if it is equivalent
to the direct sum sequence $0\to Y\to Y\oplus Z\to Z\to 0$. This
happens if and only if it splits, that is, there is an operator
$p:X\to Y$ such that $p\imath={\bf 1}_Y$ ($\imath(Y)$ is
complemented in $X$); equivalently, there is an operator $s:Z\to
X$ such that $\pi s={\bf 1}_Z$.

For every pair of Banach spaces $Z$ and $Y$, we denote by
$\Ext(Z,Y)$ the space of all exact sequences $0 \to Y \to X \to
Z\to 0$ modulo equivalence. We write $\Ext(Z,Y)=0$ to indicate
that every sequence of the form (\ref{SEX}) is trivial. The reason
for this notation is that $\Ext(Z,Y)$ has a natural linear
structure \cite{cabecastlong,castgonz} for which the (class of
the) trivial extension is the zero element.

Suppose we are given an extension (\ref{SEX}) and an operator $t:
Y\to B$. Consider the push-out of the couple $(\imath, t)$ and
draw the corresponding arrows:
\begin{equation*}%\label{equiv}
\begin{CD}
0  @>>>  Y @>\imath>> X @>>> Z @>>>0 \\
  &  &   @Vt VV  @VVt'V
   & \\
& &B @>\imath'>> \PO  \\
\end{CD}
\end{equation*}
By Lemma~\ref{isom}(a), $\imath'$ is an isomorphic embedding. Now,
the operator $\pi:X\to Z$ and the null operator $n:B\to Z$ satisfy
the identity $\pi\imath=nt=0$, and the universal property of the
push-out gives a unique operator $\varpi: \PO\to Z$ making the
following diagram commutative:
\begin{equation}\label{po-seq}
\begin{CD}
0  @>>>  Y @>\imath>> X @>\pi>> Z @>>>0 \\
  &  &   @Vt VV  @VVt'V
   @| \\
0  @>>> B @>\imath'>> \PO  @>\varpi >> Z@>>> 0
\end{CD}
\end{equation}
Or else, just take $\varpi((x,b)+\Delta)=\pi(x)$, check
commutativity, and discard everything but the definition of $\PO$.

Elementary considerations show that the lower sequence in the
preceding Diagram is exact. That sequence will we referred to as
the push-out sequence. Actually,  the universal property of the
push-out makes this diagram unique, in the sense that for any
other commutative diagram of exact sequences
\begin{equation*}%\label{po-seq}
\begin{CD}
0  @>>>  Y @>\imath>> X @>\pi>> Z @>>>0 \\
  &  &   @Vt VV  @VVV
   @| \\
0  @>>> B @>>> X'  @>>> Z@>>> 0
\end{CD}
\end{equation*}
the lower row turns out to be equivalent to the push-out sequence
in (\ref{po-seq}). For this reason we usually refer to a diagram
like that as a push-out diagram. The universal property  of the
push-out diagram immediately yields

\begin{lema}\label{crit-po} With the above notations,  the push-out sequence splits if and only if $t$ extends to $X$, that is, there is an operator $T:X\to B$ such that $T\imath=t$.
\end{lema}

Proceeding dually one obtains the pull-back sequence. Consider
again (\ref{SEX}) and an operator $u:A\to Z$. Let us form the
pull-back diagram of the couple $(\pi, u)$ thus:
\begin{equation*}%\label{equiv}
\begin{CD}
0  @>>>  Y @>\imath>> X @>\pi>> Z @>>>0 \\
  &  & & &  @A {'}\!u AA  @AA u A
    \\
&& & & \PB  @> {'}\!\pi >> A
\end{CD}
\end{equation*}
Recalling that ${'}\!\pi$ is onto and taking
$\jmath(y)=(0,\imath(y))$, it is easily seen that the following
diagram is commutative:
\begin{equation}\label{pb-seq}%\label{equiv}
\begin{CD}
0  @>>>  Y @>\imath>> X @>\pi>> Z @>>>0 \\
  &  & @|  @A{'}\!u AA  @AA u A
    \\
0@>>> Y@>\jmath >> \PB  @>{'}\!\pi >> A@>>>0\\
\end{CD}
\end{equation}
The lower sequence is  exact, and we shall referred to it as the
pull-back sequence. The splitting criterion is now as follows.

\begin{lema}\label{crit-po} With the above notations,  the pull-back sequence splits if and only if $u$ lifts to $X$, that is, there is an operator $L:A\to X$ such that $\pi L=u$.
\end{lema}

Given an exact sequence $0 \To Y \stackrel{\imath}\To X
\stackrel{\pi}\To Z \To 0$ and another Banach space $B$, taking
operators with values in $B$ one gets the exact sequence
$$%\begin{CD}
0 \To \mathfrak L(Z,B)  \stackrel{\pi^*}\To \mathfrak  L(X,B)
\stackrel{\imath^*}\To \mathfrak  L(Y,B)$$ (in which $\varkappa^*$
means composition with $\varkappa$ on the right) that can be
continued to form a ``long exact sequence"
$$
%\begin{CD}
0 \To \mathfrak  L(Z,B)  \stackrel{\pi^*}\To \mathfrak  L(X,B)
\stackrel{\imath^*}\To \mathfrak  L(Y,B) \stackrel{\beta}\To
\Ext(Z,B)\To \Ext(X,B)\To \Ext(Y,B)
%\end{CD}
$$
A detailed description of homology sequences can be seen in
\cite{cabecastlong}. Here we just indicate the action of $\beta$:
it takes $t\in \mathfrak  L(Y,B)$ into (the class in $\Ext(Z,B)$
of) the lower extension of the push-out diagram  (\ref{po-seq}).

\section{Basic properties of (universally) separably injective
spaces}

\begin{defin}
A Banach space $E$ is separably injective if
for every separable Banach space $X$ and each subspace $Y\subset
X$, every operator $t:Y\to E$ extends to an operator $T:X\to
E$.  If some extension $T$ exists with $\|T\|\leq\lambda\|t\|$ we say that $E$ is $\lambda$-separably
injective.
\end{defin}

It is easy to check that every separably injective space $E$ is
$\lambda$-separably injective for some $\lambda$ since every
sequence of norm-one operators $t_n: Y_n \to E$ induces a norm-one
operator $t:\ell_1(Y_n) \to E$. Separable injective spaces can be
characterized as follows.

\begin{prop}\label{sepinj-char}
For a Banach space $E$ the following properties are equivalent.
\begin{enumerate}
\item[(a)] $E$ is separably injective. \item[(b)] Every operator
from a subspace of $\ell_1$ into $E$ extends to $\ell_1$.
\item[(c)] For every Banach space $X$ and each subspace $Y$ such
that $X/Y$ is separable, every operator $t:Y\to E$ extends to $X$.
\item[(d)] If $X$ is a Banach space containing $E$ and $X/E$ is
separable, then $E$ is complemented in $X$. \item[(e)] For every
separable space $S$ one has $\Ext(S, E)=0$.
\end{enumerate}
Moreover, \begin{enumerate} \item The space $E$ is
$\lambda$-complemented in every $Z$ such that $Z/E$ is separable
if and only if every operator $t: Y \to E$ admits an extension
${T} : X \to E$ with $\|{T}\|\leq \lambda \|t\|$, whenever $X/Y$
is separable.

\item If $E$ is $\lambda$-separably injective, then for every
operator $t: Y \to E$ there exists an extension ${T} : X \to E$ of
$T$ with $\|{T}\|\leq 3 \lambda \|t\|$, whenever $X/Y$ is
separable.
\end{enumerate}
\end{prop}

\begin{proof} It is clear that $(c)\Rightarrow (a) \Rightarrow
(b)$ and $(c) \Rightarrow (d) \Leftrightarrow (e)$. Moreover,
$(1)$ shows that $(d) \Rightarrow (c)$ and $(2)$ shows that
$(a)\Rightarrow (c)$. The remaining implication $(b)\Rightarrow
(a)$ follows from the proof of (2) below.\\

For the sufficiency statement in (1) simply consider $t$ as the
identity on $E$. For the necessity statement, given an operator
$t: Y \to E$ form  the associated  push-out diagram$$
\begin{CD}
 0 @>>> Y @>\imath>> X @>\pi>> X/Y @>>> 0\\
 & & @V{t}VV @VV{t'}V @|\\
 0 @>>> E @>\imath'>> \PO @>>> \PO/E @>>> 0.
\end{CD}
$$ Since $\PO/E = X/Y$ is separable, there
is a projection $p: \PO \to E$ with norm at most $\lambda$, and
thus, recalling that $\|t'\| \leq 1$, the composition $p t': X \to
E$ yields an extension of $t$ with norm at most $\lambda$.

The proof for (2) is a little more tricky. Let $q$ be a surjective
map from $\ell_1 \to X/Y$. The lifting property of $\ell_1$
provides an operator $Q: \ell_1\to X$. Consider thus the
commutative diagram
$$
\begin{CD}
 0 @>>> \ker q @>\jmath>> \ell_1 @>q>> X/Y @>>> 0\\
 & & @V{\phi}VV @VQVV @| \\
 0 @>>> Y @>>> X @>>> X/Y @>>> 0
\end{CD}
$$Let us construct the true push-out of the couple $(\phi, \jmath)$ and
the corresponding complete diagram
$$
\begin{CD}
 0 @>>> \ker q @>\jmath >> \ell_1 @>q>> X/Y @>>> 0\\
 & & @V{\phi}VV @VV\phi'V @|\\
 0 @>>> Y @>\jmath'>> \PO @>>> X/Y @>>> 0.
\end{CD}
$$
We can consider without loss of generality that $\|\phi\|=1$. Let
$S: \ell _1 \to E$ be an extension of $t\phi$ with $\|S\| \leq
\lambda \|t\phi\| \leq \lambda \|t\|$. By the universal property
of the push-out, there exists an operator $L : \PO \to E$ such
that $L \phi' = S$ and $ \| L\| \leq \max \{\|t\|, \|S\|\} \leq
\lambda \|t\|. $ Again by the universal property of the push-out,
there is a diagram of equivalent exact sequences
$$
\begin{CD}
 0 @>>> Y@>\jmath'>> \PO @>>> X/Y @>>> 0\\
 & & @|  @V{\gamma}VV @|\\
 0 @>>> Y @>\imath >> X @>p>> X/Y @>>> 0,
 \end{CD}
 $$
where the isomorphism $\gamma$ is defined as $\gamma((y, u) +
\Delta) = \jmath (y) +\; Q(u)$ is such that $ \|\gamma\| \leq \max
\{\|\jmath\|, \|Q\|\} \leq 1. $ The desired extension of $t$ to
$X$ is ${T} = L \gamma^{-1}$, where $\gamma^{-1}$ comes defined by
$$
\gamma^{-1}(x) = (x- s(px), s(px)) + \Delta,
$$where $s: X/Y \to \ell_1$ is a homogeneous bounded selection for
$q$ with $\|s\| \leq 1 $. One clearly has $\|\gamma^{-1}\| \leq
3$, and therefore $\|{T}\|\leq 3 \lambda$.
\end{proof}

We are especially interested in the following subclass of
separably injective spaces.

\begin{defin} A Banach space $E$ is said to be universally separably
injective if for every Banach space $X$ and each separable
subspace $Y\subset X$, every operator $t:Y\to E$ extends to an
operator $T:Y\to X$. If some extension $T$ exists with
$\|T\|\leq\lambda \|t\|$ we say that $E$ is
universally $\lambda$-separably injective.
\end{defin}

It is easy to check that a Banach space $E$ is universally
separably injective if and only if every $E$-valued operator with
separable range extends to any superspace. It is also easy to show
that every universally separably injective space is
$\lambda$-universally separably injective for some $\lambda$.

Recall that a Banach space $X$ has Pe\l czy\'nski's property $(V)$
if each operator defined on $X$ is either weakly compact or it is
an isomorphism on a subspace isomorphic to $c_0$. We will say that
$X$ has Rosenthal's property $(V)$ if it satisfies the preceding
condition with $\ell_\infty$ replacing $c_0$. It is well-known
that Lindenstrauss\ spaces (i.e.,
$\mathcal{L}_{\infty,1+}$-spaces) have this property
\cite{johnzippre}.

Not all $\mathcal{L}_\infty$-spaces have Pe\l czy\'nski's property
$(V)$: for example, the $\mathcal L_\infty$-spaces without copies
of $c_0$ constructed by Bourgain and Delbaen \cite{BourgainD80};
or those that can be obtained from Bourgain-Pisier \cite{bp}; or
the space $\Omega$ constructed in \cite{ccky} as a twisted sum
$$
0 \To C[0,1] \To \Omega \To c_0 \To 0
$$
with strictly singular quotient map. Recall that a Banach space
$X$ is said to be a Grothendieck space if every operator from $X$
to a separable Banach space (or to $c_0$) is weakly compact.
Clearly, a Banach space with property $(V)$ is a Grothendieck
space if and only  if it has no complemented subspace isomorphic
to $c_0$. It is well-known that $\ell_\infty$ is a Grothendieck
space. One moreover has.

\begin{prop}\label{(V)}$\;$
\begin{itemize}
\item[(a)] A separably injective space is of type
$\mathcal{L}_{\infty}$, has Pe\l czy\'nski's property (V) and,
when it is infinite dimensional, contains copies of $c_0$.

\item[(b)] A universally separably injective space is a
Grothendieck space of type $\mathcal{L}_{\infty}$, has Rosenthal's
property $(V)$ and, when it is infinite dimensional, contains
$\ell_\infty$.
\end{itemize}

\end{prop}
\begin{proof}
(a) Let $E$ be a $\lambda$-separably injective space. We want to
see that if $Y$ is a subspace of any Banach space $X$, every
operator $t:Y\to E$ extends to an operator $T:X\to E^{**}$ with
$\|T\|\leq\lambda\|t\|$. This implies that $E^{**}$ is
$\lambda$-injective, by an old result of Lindenstrauss\
\cite[Theorem 2.1]{lind}. Being of infinite dimension,  $E^{**}$
is an $\mathcal{L}_{\infty,9\lambda^+}$ space and so is $E$. Let
$t:Y\to E$ be an operator. Given a finite-dimensional subspace $F$
of $X$, let $T_F:F\to E$ be any operator extending the restriction
of $t$ to $Y\cap F$. Let $\mathscr F$ be the set of
finite-dimensional subspaces of $X$, ordered by inclusion, let
$\U$ be any ultrafilter refining the Fr\'echet filter on $\mathscr
F$, that is, containing every set of the form $\{G\in\mathscr F:
F\subset G\}$ for fixed $F\in\mathscr F$. Then, define $T:X\to
E^{**}$ taking
$$
T(x)=\text{weak*-}\lim_{\U(F)}T_F(1_{F(x)}x).
$$
It is easily seen that $T$ is a linear extension of $t$, with $\|T\|\leq \lambda\|t\|$.

To show that $E$ contains $c_0$ and
has property $(V)$, let $T: E \to X$ be a non-weakly compact
operator ($E$ being an infinite dimensional $\mathcal L_\infty$ space cannot be
reflexive). Choose a bounded sequence $(x_n)$ in $E$ such that
$(Tx_n)$ has no weakly convergent subsequences and let $Y$ be the subspace spanned by $(x_n)$ in $E$. As $Y$ is separable we can regard it as a subspace of
$C[0,1]$. Let $J: C[0,1] \to E$ be any operator extending the inclusion of $Y$ into $E$. Since $TJ
:C[0,1] \to E$ is not weakly compact,  $TJ$ is an isomorphism on
some subspace isomorphic to $c_0$; and the same occurs to $T$.

(b) If, in addition to that, $E$ is universally separably
injective we may take $T:E\to Z$ and $Y\subset E$ as before but
this time we consider $Y$ as a subspace of $\ell_\infty$. If
$J:\ell_\infty\to E$ is any extension of the inclusion of $Y$ into
$E$, then $TJ :\ell_\infty\to Z$ is not weakly compact. Hence it
is an isomorphism on some subspace isomorphic to $\ell_\infty$ and
so is $T$.
\end{proof}

Several modifications on the proof of Ostrovskii \cite{ostr} yield

\begin{prop}\label{first}
 A $\lambda$-separably injective space with $\lambda
<2$ is either finite-dimensional or has density character at least $\frak c$.
\end{prop}

Recall that a class of Banach spaces is said to have the 3-space
property if whenever $X/Y$ and $Y$ belong to the class, then so
$X$ does. See the monograph \cite{castgonz}.

\begin{prop}\label{quotsi}\quad

\begin{enumerate}
\item The class of separably injective spaces has the 3-space
property.

\item The quotient of two separably injective spaces is separably
injective.

\item The class of universally separably injective spaces has the
3-space property.

\item The quotient of a universally separably injective space by a
separably injective space is universally separably injective.
\end{enumerate}
\end{prop}
\begin{proof} The simplest proof for the 3-space property (1) follows
from characterization (2) in Proposition \ref{sepinj-char}: let us
consider an exact sequence $ 0 \To F \To E \stackrel{\pi}\To G \To
0$ in which both $F$ and $G$ are separably injective. Let $\phi: K
\to E$ be an operator from a subspace $\imath :K \to \ell_1$ of
$\ell_1$; then $\pi\phi$ can be extended to an operator $\Phi:
\ell_1 \to G$, which can in turn be lifted to an operator $ \Psi:
\ell_1 \to E$. The difference $\phi - \Psi \imath$ takes values in
$F$ and can thus be extended to an operator $e: \ell_1 \to F$. The
desired operator is $ \Psi + e$. A different homological proof
that properties having the form  $\Ext(X, -)=0$ are always 3-space
properties can be found in \cite{cabecastlong}.

To prove (2) and (4) let us consider an exact sequence $ 0 \To F \To E \stackrel{\pi}\To
G \To 0$ in which $F$ is
separably injective and $E$ is (universally) separably
injective. Let $\phi: Y \to G$ be an operator from a separable
space $Y$ which is a subspace of a separable (arbitrary)
space $X$. Consider the pull-back diagram
$$
\begin{CD} 0@>>> F @>>> E @>q>> G @>>> 0\\
& & @| @AA{\Phi}A @AA{\phi}A\\
 0@>>> F@>>> \PB @>Q>>
Y@>>> 0\end{CD}
$$
Since $F$ is separably injective, the lower exact sequence splits,
so $Q$ has a selection operator $s:Y\to \PB$. By the injectivity
assumption about $E$, there exists an operator $T:X\to E$ agreeing
with $Q s$ on $Y$. Then $q T:X\to G$ is the desired extension of
$\phi$.

The proof for (3) has to wait until Theorem \ref{Thetastar=Sinfty}
when a suitable characterization of universally separably
injective spaces will be presented.
\end{proof}

Several variations of these results can be seen in
\cite{castmorestud}. It is obvious that if $(E_i)_{i\in I}$ is a
family of $\lambda$-separably injective Banach spaces, then
$\ell_\infty(I,E_i)$ is $\lambda$-separably injective. The
non-obvious fact that also $c_0(I,E_i)$ is separably injective can
be considered as a vector valued version of Sobczyk's theorem.
Proofs for this result have been obtained by Johnson-Oikhberg
\cite{johnoikh}, Rosenthal \cite{roseco}, Cabello \cite{cabeco}
and Castillo-Moreno  \cite{castmoresob}, each with its own
estimate for the constant. These are $2\lambda^2$ (implicitly),
$\lambda(1+\lambda)^+, (3\lambda^2)^+$ and $6\lambda^+$,
respectively.

\begin{remark}\label{bourgremark}
Let $0\to F\to E\to G\to0$ be a short exact sequence of Banach
spaces. We know from Proposition~\ref{quotsi} that $E$ is
separably injective if the other two relevant spaces are; and the
same happens with $G$. What about $F$?
 Bourgain \cite{bour} constructed an uncomplemented copy
of $\ell_1$ in $\ell_1$, which yields an exact sequence
 $0 \to \ell_1 \to
\ell_1 \to B \to 0$ that does not split. By Lindenstrauss{'}
lifting $B$ is not an $\mathcal L_1$ space. Its dual sequence $0
\to B^* \to \ell_\infty \to \ell_\infty \to 0$ shows that the
kernel of a quotient mapping between two injective spaces may fail
to be even an $\mathcal L_\infty$-space.

\end{remark}

%\newpage

\section{Examples}\label{construct-sepinj}

All injective spaces are universally separably injective. Sobczyk
theorem states that $c_0$ --and $c_0(\Gamma)$, in general-- are
2-separably injective in its natural supremum norm. They are not
universally separably injective since they do not contain
$\ell_\infty$.

\subsection{Twisted sums} The $3$-space property yields that twisted sums of
separably injective are also separably injective. In particular:
\begin{itemize} \item Twisted sums of $c_0$ and $c_0(\Gamma)$.
This includes the Johnson-Lindenstrauss spaces $C(\Delta_\mathcal
M)$ \cite{johnlind} obtained taking the closure of the linear span
in $\ell_\infty$ of the characteristic functions $\{1_{n}\}_{n \in
\N}$ and $\{1_{M_\alpha}\}_{\alpha\in J}$ for an uncountable
almost disjoint family $\{M_\alpha\}_{\alpha\in J}$ of subsets of
$\N$. Marciszewski and Pol answer in \cite{marcpol} a question of
Koszmider \cite[Question 5]{koszpams} showing that there exist
$2^{\mathfrak c}$ almost disjoint families $\mathcal M$ generating
non-isomorphic $C(\Delta_{\mathcal M})$-spaces.

\item Twisted sums of two nonseparable $c_0(\Gamma)$ spaces. This
includes variations of the previous construction using the
Sierpinski-Tarski \cite{kuramos} generalization of the
construction of almost-disjoint families; the Ciesielski-Pol space
(see \cite{devilgod}); the WCG nontrivial twisted sums of
$c_0(\Gamma)$ obtained independently by Argyros, Castillo,
Granero, Jimenez and Moreno \cite{acgjm} and by Marciszewski
\cite{marc}

\item Twisted sums of $c_0$ and $\ell_\infty$, as those
constructed in \cite{cabecastuni}.

\item A twisted sum of $c_0$ and $c_0(\ell_\infty/c_0)$ that is
not complemented in any $C(K)$-space, as the one obtained in
\cite{castmorestud}.
\end{itemize}

It is not hard to prove that none of the preceding examples can be
universally separably injective.

\subsection{The space $\ell_\infty^c(\Gamma)$}  A typical $1$-universally separably injective space is the
space $\ell_\infty^c(\Gamma)$ of countably supported bounded
functions $f:\Gamma\to\R$, where $\Gamma$ is an uncountable set.
This space is isomorphic but not isometric to some $C(K)$ space,
showing in this way that the theory of universally separably
injective spaces does not run parallel with that of injective
spaces. What makes this space universally separably injective
space is the following property:

\begin{defin}
We say that a Banach space $X$ is $\ell_\infty$-upper-saturated if
every separable subspace of $X$ is contained in some (isomorphic)
copy of $\ell_\infty$ inside $X$.
\end{defin}

It is clear that an $\ell_\infty$-upper-saturated space is
universally separably injective. We will prove in Theorem
\ref{Thetastar=Sinfty} that the converse also holds.

\subsection{The space $\ell_\infty/c_0$}  Since $\ell_\infty$ is injective and $c_0$ is separably
injective, it follows from Proposition \ref{quotsi} that
$\ell_\infty/c_0$ is universally separably injective, although the
constant is not optimal. It follows from Proposition
\ref{cor:mideal}(a) that $\ell_\infty/c_0$ is 1-universally
separably injective, hence --by Theorem \ref{Thetastar=Sinfty}--
it is $\ell_\infty$-upper-saturated. This can be improved to show
that every separable subspace of $\ell_\infty/c_0$ is contained in
a subalgebra of $\ell_\infty/c_0$ isometrically isomorphic to
$\ell_\infty$.

It is well-known that $\ell_\infty/c_0$ is not injective. The
simplest proof appears in Rosenthal \cite{disjoint}: an injective
space containing $c_0(I)$ must also contain $\ell_\infty(I)$; it
is well-known that  $\ell_\infty/c_0$ contains $c_0(I)$ for $|I|=
\mathfrak c$ while it cannot contain $\ell_\infty(I)$. The proof
is is quite rough in a sense: it says that $\ell_\infty/c_0$ is
uncomplemented in its bidual, a huge superspace. Denoting
$\N^*=\beta \N \setminus \N$, Amir had shown in \cite{A} that
$C(\mathbb{N}^*)$ is not complemented in $\ell_\infty(2^\mathfrak
c)$, which provides another proof that $l_\infty/c_0$ is not
injective. Amir's proof can be refined in order to get
$C(\mathbb{N}^\ast)$ uncomplemented in a much smaller space. It
can be shown \cite{castplic} that $C(\mathbb{N}^\ast)$ contains an
uncomplemented copy $Y$ of itself. Corollary~\ref{unexpected}
yields that the corresponding quotient $C(\N^*)/ Y$ cannot be
isomorphic to a quotient of $\ell_\infty$ by a separable subspace.

\subsection{Other $C(K)$-spaces}\label{sub:CKF}

Recall that a compact Hausdorff space $K$ is said to be an
$F$-space if disjoint open $F_\sigma$ sets (equivalently,
cozeroes) have disjoint closures. Equivalently, if any continuous
function $f:K\to \R$ can be decomposed as $f= u|f|$ for some
continuous function $u: K\to \R$. One has (see \cite{accgm3} for a
proof of this and several generalized forms of this result).

\begin{prop} \label{people-sep}
A $C(K)$ space is  $1$-separably injective if and only if $K$ is
an $F$-space.
\end{prop}

Simple examples show that when a $C(K)$-space is only isomorphic
to a $1$-separably injective then $K$ does not need to be an
$F$-space. It is an immediate consequence of Tietze's extension
theorem that a closed subset of an $F$-space is an $F$-space. In
particular, $\N^*=\beta \N \setminus \N$ is an $F$-space.

Given a compact space $K$, we write  $K'$ for its derived set,
that is, the set of its accumulation points. This process can be
iterated to define $K^{(n+1)}$ as $(K^{(n)})'$ with $K^{(0)}=K$.
We say that $K$ has height $n$ if $K^{(n)}=\varnothing$. We say
that $K$ has finite height if it has height $n$ for some $n\in\N$.

\begin{prop}\label{finiteith} If $K$ is a compact space of height $n$, then $C(K)$ is
$(2n-1)$-separably injective. Consequently, if $K$ is a compact
space of finite height then $C(K)$ is separably injective although
it is not universally separably injective.
\end{prop}
\begin{proof} Let $Y \subset X$ with $X$ separable and let $t: Y \to C(K)$ be a norm one operator.
The range of $t$ is separable and every separable subspace of a
$C(K)$ is contained in an isometric copy of $C(L)$, where $L$ is
the quotient of $K$ after identifying $k$ and $ k'$ when
$y(k)=y(k')$ for all $y \in Y$. This $L$ is metrizable because $Y$
is separable. Moreover, if $K$ has height $n$, then $L$ has height
at most $n$ and so it is homeomorphic to $[0,\omega^r\cdot k]$
with $r<n$, $k<\omega$. Since $C[0,\omega^r\cdot k]$ is
$({2r+1})$-separably injective \cite{bake},  our operator can be
extended to an operator ${T}: X \to C(K)$ with norm
$$
\|{T}\|\leq (2r+1)\|t\| \leq (2n-1)\|t\|,
$$
concluding the proof.\end{proof}

When $K$ is a metrizable compact of finite height $n$,
Baker~\cite{bake} showed that $2n-1$ is the best constant for
separable injectivity, using arguments from Amir~\cite{A}. There
are some difficulties in generalizing those arguments for
nonmetrizable compact spaces, so we do not know if it could exist
a nonmetrizable compact space $K$ of height $n$ such that $C(K)$
is $\lambda$-separably injective for some $\lambda<2n-1$.

\begin{prop} The space of all bounded Borel (respectively, Lebesgue)
measurable functions on the line is 1-separably injective in the
sup norm.\end{prop}

\begin{proof}
Clearly the given spaces are in fact Banach algebras satisfying
the inequality required by Albiac-Kalton characterization (see
\cite{albiac-kalton-book}). Thus they can be represented as $C(K)$
spaces. On the other hand, each measurable function can be
decomposed as $f=u|f|$, with $u$ (and $|f|$, of course)
measurable. This clearly implies that the corresponding compacta
are $F$-spaces.
\end{proof}

Argyros proved in \cite{argyros83} that none of the spaces in the
above example is injective. This is very simple in the Borel case:
the characteristic functions of the singletons generate a copy of
$c_0(\R)$ in the space of bounded Borel functions. The density
character of the latter space is the continuum, as there are
$\mathfrak c$ Borel subsets. Therefore it cannot contain a copy of
$\ell_\infty(\R)$, whose density character is $2^\mathfrak c$.\\

\subsection{M-ideals} A closed subspace
$J\subset X$ is called an $M$-ideal \cite[Definition 1.1]{hww} if
its annihilator $J^\perp=\{x^*\in X^*: \langle x^*,x\rangle=0\quad
\forall x\in J\}$ is an $L$-summand in $X^*$. This just means that
there is a linear projection $P$ on $X^*$ whose range is $J^\perp$
and such that $\|x^*\|=\|P(x^*)\|+\|x^*-P(x^*)\|$ for all $x^*\in
X^*$. The easier examples of $M$-ideals are just ideals in
$C(K)$-spaces. In particular, if $M$ is a closed subset of the
compact space $K$ and $L=K\setminus M$ one has that $C_0(L)$ is an
$M$-ideal in $C(K)$ is straightforward from the Riesz
representation of $C(K)^*$. A remarkable generalization of
Borsuk-Dugundji theorem for $M$-ideals was provided by Ando
\cite{ando} and, independently, Choi and Effros
\cite{choi-effros}. In order to state it let us recall that a
Banach space $Z$ has the $\lambda$-approximation property
($\lambda$-AP, for short) if, for every $\varepsilon > 0$ and
every compact subset $K$ of $Z$, there exists a finite rank
operator $T$ on $Z$, with $\|T\|\leq \lambda$, such that
$\|Tz-z\|< \varepsilon$, for every $z\in K$. We say that $Z$ has
the bounded approximation property (BAP for short) if it has the
$\lambda$-AP, for some $\lambda$.

We refer the reader to \cite{casazza} for background and basic
information about approximation properties and, in particular, the
BAP and the uniform approximation property (UAP in short).

\begin{teor}[\cite{hww}, Theorem 2.1]\label{mideal}

Let $J$ be an $M$-ideal in the Banach space $E$ and $\pi:E\to E/J$
the natural quotient map. Let $Y$ be a separable Banach space and
$t:Y\to E/J$ be an operator.  Assume further that one of the
following conditions is satisfied:
\begin{enumerate}
\item $Y$ has the $\lambda$-AP. \item $J$ is a Lindenstrauss\
space.
\end{enumerate}
Then $t$ can be lifted to $E$, that is, there is an operator $T: Y
\to E$ such that $ \pi T = t$. Moreover one can get $\|T\|\leq
\lambda\|t\|$ under the assumption {\rm(1)} and $\|T\|=\|t\|$
under {\rm(2)}.
\end{teor}

One has.
\begin{prop}\label{cor:mideal} Let $J$ be an $M$-ideal in a Banach space $E$.
\begin{itemize}
\item[(a)] If $E$ is $\lambda$-(universally) separably injective, then  $E/J$ is $\lambda^2$-(universally) separably injective.
\item[(b)] If $E$ is $\lambda$-separably injective, then $J$ is $2\lambda^2$-separably injective.
\end{itemize}

When $J$  is a Lindenstrauss\ space (which is always the case if
$E$ is), then the exponent 2 disappears. In particular, if $K_1$
is a closed subset of the compact space $K$ and $K_0=K\setminus
K_1$ one has:
\begin{itemize}
\item[(c)] If $C(K)$ is $\lambda$-(universally) separably
injective, then so is $C(K_1)$. \item[(d)] If $C(K)$ is
$\lambda$-separably injective, then $C_0(K_0)$ is
$2\lambda$-separably injective.
\end{itemize}

\end{prop}
\begin{proof}
(a) By (the proof of) Proposition~\ref{(V)}, $E^{**}$ is
$\lambda$-injective and so it has the $\lambda$-AP. Since
$E^{**}=J^{**}\oplus_\infty (E/J)^{**}$  we see that also $J^{**}$
and $(E/J)^{**}$ have the $\lambda$-AP. Hence both $J$ and $(E/J)$
have the  $\lambda$-AP. Let $Y$ be a separable subspace of $X$ and
$t:Y\To E/J$ an operator. Let $S$ be a separable subspace of $E/J$
containing the image of $t$. By \cite[Theorem 9.7]{casazza} we may
assume $S$ has the $\lambda$-AP. Let $s: S\To E$ be the lifting
provided by Theorem~\ref{mideal}, so that $\|s\|\leq \lambda$.
Now, if $T:X\To E$ is an extension of $st$, then $\pi T:X\To E/J$
is an extension of $t$, and this can be achieved with $\|\pi
T\|=\|T\|\leq \lambda^2 \|t\|$.

(d) --and (b)--. Let us remark that if $S$ is a subspace of $C(K)$
containing $C_0(K_0)$ and $S/C_0(K_0)$ is separable, then there is
a projection $p:S\To C_0(K_0)$ of norm at most 2. Indeed,
$S/C_0(K_0)$ is a separable subspace of $C(K_1)$ and there is a
lifting $s: S/C_0(K_0)\To C(K)$, with $\|s\|=1$, and ${p=\bf
1}_S-sr$ is the required projection. Now,  let $t:Y\To C_0(K_0)$
be an operator, where $Y$ is a subspace of a separable Banach
space $X$. Considering $t$ as taking values in $C(K)$, there is an
extension $T:X\To C(K)$ with $\|T\|\leq\lambda\|t\|$. Let $S$
denote the least closed subspace of $C(K)$ containing the range of
$T$ and $C_0(K_0)$ and $p:S\To C_0(K_0)$ a projection with
$\|p\|\leq 2$. The composition $pT:X\To C_0(K_0)$ is an extension
of $t$ and clearly, $\|pT\|\leq 2\lambda\|t\|$.
\end{proof}

\subsection{Ultraproducts of type $\mathcal L_\infty$}\label{ch:ultraproducts}
Let us briefly recall the definition and some basic properties of
ultraproducts of Banach spaces. For a detailed study of this
construction at the elementary level needed here we refer the
reader to Heinrich's survey paper \cite{heinrich} or Sims' notes
\cite{sims}. Let $I$ be a set, $\U$ be an ultrafilter on $I$, and
$\Xii$ a family of Banach spaces. Then $ \ell_\infty(X_i)$ endowed
with the supremum norm, is a Banach space, and $ c_0^\U(X_i)=
\{(x_i) \in \ell_\infty(X_i) : \lim_{\U(i)} \|x_i\|=0\} $ is a
closed subspace of $\ell_\infty(X_i)$. The ultraproduct of the
spaces $\Xii$ following $\U$ is defined as the quotient
$$
[X_i]_\U = {\ell_\infty(X_i)}/{c_0^\U(X_i)}.
$$
We denote by $[(x_i)]$ the element of $[X_i]_\U$ which has the
family $(x_i)$ as a representative. It is not difficult to show
that $ \|[(x_i)]\| = \lim_{\U(i)} \|x_i\|. $ In the case $X_i = X$
for all $i$, we denote the ultraproduct by $X_\U$, and call it the
ultrapower of $X$ following $\U$. If $T_i:X_i\to Y_i$ is a
uniformly bounded family of operators, the ultraproduct operator
$[T_i]_\U: [X_i]_\U\to [Y_i]_\U$ is given by $[T_i]_\U[(x_i)]=
[T_i(x_i)]$. Quite clearly, $ \|[T_i]_\U\|= \lim_{\U(i)}\|T_i\|. $

\begin{defin}\label{CI}
An ultrafilter $\U$ on a set $I$ is countably incomplete if there
is a decreasing sequence $(I_n)$ of subsets of $I$ such that
$I_n\in \U$ for all $n$, and $\bigcap_{n=1}^\infty
I_n=\varnothing$.
\end{defin}

Notice that $\U$ is countably incomplete if and only if there is a
function $n:I\to \N$ such that $n(i)\to\infty$ along $\U$
(equivalently, there is a family $\e(i)$ of strictly positive
numbers converging to zero along $\U$). It is obvious that any
countably incomplete ultrafilter is non-principal  and also that
every non-principal (or free) ultrafilter on $\N$ is countably
incomplete. Assuming all free ultrafilters  countably incomplete
is consistent with {\textsf{ZFC}}, since the cardinal of a set
supporting a free countably complete ultrafilter should
be measurable, hence strongly inaccessible.\\

It is clear that the classes of $\mathcal{L}_{p,\lambda^+}$ spaces
are stable under ultraproducts \cite[Proposition
1.22]{BourgainLNM}. In the opposite direction, a Banach space is a
$\mathcal{L}_{p,\lambda^+}$ space if and only if some (or every)
ultrapower is. In particular, a Banach space is an
$\mathcal{L}_{\infty}$ space or a Lindenstrauss\  space if and
only if so are its ultrapowers; see, e.g., \cite{heinrichL1}. It
is possible however to produce a Lindenstrauss\ space out of
non-even-$\mathcal L_\infty$-spaces: indeed, if $p(i)\to\infty$
along $\U$, then the ultraproduct $[L_{p(i)}]_\U$ is a
Lindenstrauss\ space (and, in fact, an abstract $\mathcal
M$-space; see \cite[Lemma 3.2]{c:M}).

The following result about the structure of separable subspaces of
ultraproducts of type $\mathcal L_\infty$ will be fundamental for
us.

\begin{lema}\label{S}
Supppose $\XiU$ is an $\mathcal L_{\infty, \lambda^+}$-space. Then
each separable subspace of $\XiU$ is contained in a subspace of
the form $[F_i]_\U$, where $F_i\subset X_i$ is finite dimensional
and $\lim_{\U(i)}d(F_i,\ell_\infty^{k(i)})\leq \lambda$, where
$k(i)=\dim F_i$.
\end{lema}

\begin{proof} Let us assume $S$ is an infinite-dimensional separable subspace of $\XiU$.
Let $(s^n)$ be a linearly independent sequence spanning a dense
subspace in $S$ and, for each $n$, let $(s_i^n)$ be a fixed
representative of $s^n$ in $\ell_\infty(X_i)$. Let
$S^n=\span\{s^1,\dots,s^n\}$. Since $\XiU$ is an $\mathcal
L_{\infty,\lambda^+}$-space there is, for each $n$, a finite
dimensional $F^n\subset \XiU$ containing $S^n$ with
$d(F^n,\ell_\infty^{\dim F^n })\leq \lambda+1/n$.

For fixed $n$, let $(f^m)$ be a basis for $F^n$ containing
$s^1,\dots, s^n$. Choose representatives $(f_i^m)$ such that
$f_i^m=s_i^\ell$ if $f^m=s^\ell$. Moreover, let $F_i^n$ be the
subspace of $X_i$ spanned by $f_i^m$ for $1\leq m\leq \dim F^n$.

Let $(I_n)$ be a decreasing sequence of subsets $I_n\in \U$ such
that  $\bigcap_{n=1}^\infty I_n=\varnothing$. For each integer $n$
put$$ J'_n=\{i\in I: d(F_i^n, \ell_\infty^{\dim F^n}) \leq
\lambda+2/n \}\cap I_n
$$
and $J_m=\bigcap_{n\leq m}J'_n$. All these sets are in $\U$.
Finally, set $J_\infty=\bigcap_{n}J_n$. Next we define a function
$k:I\to \N$. Set
$$
k(i)=\begin{cases}1& i\in J_\infty\\
\sup\{n:i\in J_n\}& i\notin J_\infty\\
\end{cases}
$$For each $i\in I$, take $F_i=F_i^{k(i)}$. This is a
finite-dimensional subspace of $X_i$ whose Banach-Mazur distance
to the corresponding $\ell_\infty^k$ is at most $\lambda+2/k(i)$.
It is clear that $[F_i]_{\mathscr U}$ contains $S$ and also that
$k(i)\to\infty$ along $\U$, which completes the proof.
\end{proof}

\begin{teor}\label{ultra-Linf}
Let $\Xii$ be a family of Banach spaces such that $\XiU$ is a
$\mathcal L_{\infty,\lambda^+}$-space. Then $\XiU$ is
$\lambda$-universally separably injective.
\end{teor}
\begin{proof}
It is clear that a Banach space is $\lambda$-universally separably
injective if and only if every separable subspace is contained in
some larger $\lambda$-universally separably injective subspace. By
the previous lemma, everything that has to be proved is:

\begin{lema}\label{kinf}
For every function $k:I\to\N$, the space
$[\ell_\infty^{k(i)}]_{\U}$ is 1-universally separably injective.
\end{lema}

\begin{proof} Let $\Gamma$ be the disjoint union of the sets $\{1,2,\dots, k(i)\}$ viewed as a discrete set.
Now observe that $c_0^\U(\ell_\infty^{k(i)})$ is an ideal in
$\ell_\infty(\ell_\infty^{k(i)})=\ell_\infty(\Gamma)=C(\beta\Gamma)$
and apply Proposition~\ref{cor:mideal}(a).
\end{proof}
\end{proof}

\begin{cor}\label{lind2-sep}
Let $(X_i)_{i\in I}$ be a family of Banach spaces. If $[X_i]_\U$
is a Lindenstrauss\ space, then it is $1$-universally separably
injective.
\end{cor}

\begin{remark} Ultraproducts of type $\mathcal
L_\infty$ are universally separably injective, while an infinite
dimensional ultraproduct via a countably incomplete ultrafilter is
never injective (see \cite[Theorem 2.6]{hensonmoore}; and also
\cite[ Section~8]{sims}.
\end{remark}

\section{Two characterizations of universally separably injective spaces}\label{usispaces}
In Proposition~\ref{(V)} (b) it was proved that universally
separably injective spaces contain $\ell_\infty$. Much more is
indeed true:

\begin{teor}\label{Thetastar=Sinfty}
An infinite-dimensional Banach space is universally separably injective if and only if
it is $\ell_\infty$-upper-saturated.
\end{teor}
\begin{proof}
The sufficiency is a consequence of the injectivity of
$\ell_\infty$. In order to show the necessity, let $Y$ be a
separable subspace of a universally separably injective space $X$.
We consider a subspace $Y_0$ of $\ell_\infty$ isomorphic to $Y$
and an isomorphism $t:Y_0\to Y$. We can find projections $p$ on
$X$ and $q$ on $\ell_\infty$ such that $Y\subset \ker p,
Y_0\subset \ker q$, and both $p$ and $q$ have  range  isomorphic
to $\ell_\infty$.

Indeed, let $\pi:X\to X/Y$ be the quotient map. Since $X$ contains
$\ell_\infty$ and $Y$ is separable, $\pi$ is not weakly compact
so, by Proposition~\ref{(V)}(b), there exists a subspace $M$ of $X$
isomorphic to $\ell_\infty$ where $\pi$ is an isomorphism.
Now $X/Y = \pi(M)\oplus N$, with $N$ a closed subspace. Hence $X = M \oplus \pi^{-1}(N)$, and it is enough to take  $p$ as the
projection with range $M$ and kernel $\pi^{-1}(N)$.

Since $\ker p$ and $\ker q$ are universally separably injective
spaces, we can take operators $u:X\to \ker q$ and $v:\ell_\infty\to
\ker p$ such that $v= t$ on $Y_0$ and $u = t^{-1}$ on $Y$.

Let $w:\ell_\infty \to \ran p$ be an operator satisfying
$\|w(x)\|\geq\|x\|$ for all $x\in\ell_\infty$. We will show that the operator
$$
T = v + w(\mathbf{1}_{\ell_\infty}-uv): \ell_\infty \To X
$$
is an isomorphism (into). This suffices to end the proof since $\ran T$
is isomorphic to $\ell_\infty$ and both $T$ and $v$ agree with $t$ on $Y_0$, so
 $Y\subset \ran {T} \subset X$.

Since $\ran v\subset \ker p$ and $\ran w\subset \ran p$, there exists
$C>0$ such that
$$
\|Tx\| \geq C \max\{ \|v(x)\|, \|w(\mathbf{1}_{\ell_\infty}-uv)x\|\}\quad\quad(x\in \ell_\infty).
$$
Now, if $\|vx\|< (2\|u\|)^{-1}\|x\|$, then $\|uvx\| < {1\over 2}\|x\|$;
hence
$$
\|w(\mathbf{1}_{\ell_\infty}-uv)x\|\geq \|(\mathbf{1}_{\ell_\infty}-uv)x\| > \frac{1}{2}\|x\|.
$$
Thus $\|Tx\| \geq C (2\|u\|)^{-1}\|x\|$ for every $x\in X$.
\end{proof}

We can now complete the proof of Proposition~\ref{quotsi}(3) and
show that the class of universally separably injective spaces has
the 3-space property.
\begin{prop}\label{3spu} The class of universally separably injective spaces has
the 3-space property.
\end{prop}
\begin{proof} By  Theorem~\ref{Thetastar=Sinfty} one has to show
that being $\ell_\infty$-upper-saturated is a 3-space property.

Let $0\To Y \To X \stackrel{q}\To Z \To 0$ be an exact sequence in
which both $Y,Z$ are $\ell_\infty$-uppersaturated, and let $S$ be
a separable subspace of $X$. It is not hard to find separable
subspaces $S_0, S_1$ of $X$ such that $S\subset S_1$ and $S_1/S_0
= [q(S)]$. Let $Y_\infty$ be a copy of $\ell_\infty$ inside $Y$
containing $S_0$. By the injectivity of $\ell_\infty$, $S$ is
contained in the subspace $Y_\infty \oplus [q(S)]$ of $X$. And
since there exists a copy $Z_\infty$ of $\ell_\infty$ containing
$[q(S)]$, $S$ is therefore contained in the subspace $Y_\infty
\oplus Z_\infty$ of $X$, which is isomorphic to $\ell_\infty$.
\end{proof}

A homological characterization of universally separably injective
spaces is also possible. We need first to  show:

\begin{prop}\label{elluniv} If $U$ is a universally separably injective
space then $\Ext(\ell_\infty,U)=0.$
\end{prop}
\begin{proof} James's well known distortion theorem for $\ell_1$ (resp. $c_0$)
asserts that a Banach space containing a copy of $\ell_1$ (resp.
$c_0$) also contains an almost isometric copy of $\ell_1$ (resp.
$c_0$). Not so well known is Partington's distortion theorem for
$\ell_\infty$ \cite{part}: a Banach space containing $\ell_\infty$
contains an almost isometric copy of $\ell_\infty$ (see also
Dowling \cite{dowl}). This last copy will therefore be, say,
$2$-complemented.

Let $\Gamma$ denote the set of all the $2$-isomorphic copies of
$\ell_\infty$ inside $\ell_\infty$. For each $E \in\Gamma$ let
$\imath_E: E \to \ell_\infty$ be the canonical embedding, $p_E$ a
projection onto $E$ of norm at most 2 and $u_E : E \to \ell_\infty$ an
isomorphism. Assume that a nontrivial exact sequence
$$%\begin{equation}\label{copy}\begin{CD}
 0 \To U\To X \To\ell_\infty \To 0%\end{CD}
$$%\end{equation}
exists.
We consider, for each $E\in \Gamma$, a copy of the preceding sequence,
and form the product of all these copies
$0 \To \ell_\infty(\Gamma, U)\To \ell_\infty(\Gamma, X)\To \ell_\infty(\Gamma, \ell_\infty)\To 0$.% \end{CD}$$
Let us consider the embedding $J: \ell_\infty \to
\ell_\infty(\Gamma, \ell_\infty)$ defined as $J(x)(E) =
u_E p_E(x)$ and then form the pull-back sequence
$$\begin{CD}
0 @>>> \ell_\infty(\Gamma, U)@>>> \ell_\infty(\Gamma, X)
@>>> \ell_\infty(\Gamma, \ell_\infty) @>>> 0\\
&&@| @AAA @AAJA \\
0 @>>> \ell_\infty(\Gamma, U)@>>> \PB @>q>> \ell_\infty @>>> 0
\end{CD}$$
Let us show that $q$ cannot be an isomorphism on a copy of
$\ell_\infty$. Otherwise, it would be an isomorphism on some $E
\in \Gamma$ and thus the new pull-back sequence
$$\begin{CD}
0 @>>> \ell_\infty(\Gamma, U)@>>> \PB @>q>> \ell_\infty @>>> 0\\
&&@| @AAA @AA{\imath_E}A \\
0 @>>> \ell_\infty(\Gamma, U)@>>> \PB_E @>>> E @>>> 0
\end{CD}$$
would split. And therefore the same would be true making push-out
with the canonical projection $\pi_E: \ell_\infty(\Gamma, U) \to U$
onto the $E$-th copy of $U$:
$$\begin{CD}
0 @>>> \ell_\infty(\Gamma, U)@>>> \PB_E @>>> E @>>>
0\\
&&@V{\pi_E}VV @VVV @|\\
0 @>>> U @>>> \PO_E @>>> E @>>> 0
\end{CD}$$But it is not hard to see that new pull-back with
$u_E^{-1}$
$$\begin{CD}
0 @>>> U @>>> \PO_E @>>> E @>>> 0\\
&&@| @AAA @AA{u_E^{-1}}A \\
0 @>>> U @>>> X @>>> \ell_\infty @>>> 0
\end{CD}$$produces exactly the starting sequence
which, by assumption, was nontrivial.

However, the space $\PB$ should be universally separably injective
by Proposition \ref{quotsi}(3), hence it must have Rosenthal's
property $(V)$, by Proposition \ref{(V)}(b). This contradiction
shows that the starting nontrivial  sequence cannot exist.
\end{proof}

We are thus ready to prove:

\begin{teor}\label{carauniv} A Banach space $U$ is universally separably injective
if and only if for every separable space $S$ one has
$\Ext(\ell_\infty/S, U)=0$.
\end{teor}

\begin{proof} Let $S$ be separable and let $U$ be universally separably
injective. Applying $\mathfrak L(-,U)$ to the sequence $0\to S \to
\ell_\infty \to \ell_\infty/S \to 0$ one gets the exact sequence
$$\dots\To \mathfrak L(\ell_\infty, U) \To  \mathfrak  L(S, U)\To \Ext(\ell_\infty/S, U)\To \Ext(\ell_\infty,
U)$$ Since $\Ext(\ell_\infty, U)=0$, one obtains that every exact
sequence $0 \to U \to X \to \ell_\infty/S \to 0$ fits in a
push-out diagram
$$\begin{CD} 0 @>>> S @>>> \ell_\infty @>>>
\ell_\infty/S @>>> 0\\
&& @VVV @VVV @|\\
0 @>>> U @>>> X @>>> \ell_\infty/S @>>> 0.
\end{CD}$$ Since $U$ is universally separably injective, the lower
sequence splits.

The converse is clear: every operator $t: S \to
U$ from a separable Banach space into a space $U$ produces a
push-out diagram
$$\begin{CD} 0 @>>> S @>>> \ell_\infty @>>>
\ell_\infty/S @>>> 0\\
&& @VtVV @VVV @|\\
0 @>>> U @>>> \PO @>>> \ell_\infty/S @>>> 0.
\end{CD}$$ The lower sequence splits by the assumption $\Ext(\ell_\infty/S,
U)=0$ and so $t$ extends to $\ell_\infty$, according to the
splitting criterion for push-out sequences.
\end{proof}

Which leads to the unexpected:

\begin{cor}\label{unexpected} $\Ext( \ell_\infty/c_0, \ell_\infty/c_0)=0$; i.e., every short
exact sequence $0\to \ell_\infty/c_0\to X\to \ell_\infty/c_0\to 0$
splits.
\end{cor}

This result provides a new solution for equation $\Ext(X, X)=0$.
The other three previously known types of solutions are: $c_0$ (by
Sobczyk theorem), the injective spaces (by the very definition)
and the $L_1(\mu)$-spaces (by Lindenstrauss' lifting).

Also, in rough contrast with Proposition \ref{3spu}, one has:

\begin{cor} Rosenthal's property (V) is not a 3-space
property\end{cor}
\begin{proof} With the same construction as above, start with a
nontrivial exact sequence $0 \to \ell_2 \to E \to
\ell_\infty \to 0$ (see \cite[Section 4.2]{cabecastuni}) and
construct an exact sequence
$$0 \To \ell_\infty(\Gamma, \ell_2) \To X \stackrel{q}\To \ell_\infty
\To 0,$$
 where $q$ cannot be an isomorphism on a copy of
$\ell_\infty$. So $X$ has not Rosenthal's property $(V)$. The
space $\ell_\infty(\Gamma, \ell_2)$ has Rosenthal's property $(V)$
as a quotient of $ \ell_\infty(\Gamma, \ell_\infty)=\ell_\infty
(\N\times \Gamma)$, since the property obviously passes to
quotients.
\end{proof}

It is not however true that $\Ext(U,V)=0$ for all universally
separably injective spaces $U$ and $V$ as any exact sequence $0\to
U\to \ell_\infty(\Gamma)\to \ell_\infty(\Gamma)/U \to 0$ in which
$U$ is a universally separably injective non-injective space
shows.

\section{1-separably injective spaces}\label{special}

Our first  result here establishes a major difference between
$1$-separably injective and general separably injective spaces:
$1$-separably injective spaces must be Grothendieck (hence they
cannot be separable or WCG) while a $2$-separably injective space,
such as $c_0$, can be even separable. The following lemma due to
Lindenstrauss\ \cite[p. 221, proof of (i) $\Rightarrow$
(v)]{lindpams} provides a quite useful technique.

\begin{lema}\label{aleph1sepiny}
Let $E$ be a 1-separably injective space and $Y$ a separable subspace of $X$, with $\dens X=\aleph_1$. Then
every operator $t:Y\to E$ can be extended
to an operator $T:X\to E$ with the same norm.
\end{lema}

This yields

\begin{prop}\label{1-sepinj-CH=univ} Under {\sf CH} every 1-separably injective Banach space is
universally 1-separably injective and therefore a Grothendieck space.
\end{prop}

\begin{proof} Let $E$ be 1-separably injective, $X$ an arbitrary Banach space and $t:Y\to E$ an operator, where $Y$ is a separable subspace of $X$. Let $[t(Y)]$ be the closure of the image of $t$. This is a separable subspace of $E$ and so there is an isometric embedding $u:[t(Y)]\to \ell_\infty$. As $\ell_\infty$ is 1-injective there is an operator $T:X\to\ell_\infty$ whose restriction to $Y$ agrees with $ut$. Thus it suffices to extend the inclusion of $[t(Y)]$ into $E$ to $\ell_\infty$. But, under {\sf CH}, the density character
of $\ell_\infty$ is $\aleph_1$ and the preceding  Lemma applies.
The `therefore' part is now a consequence of Proposition~\ref{(V)}(b).
\end{proof}

The ``therefore'' part survives in {\sf ZFC}:

\begin{teor}\label{1vsG}
Every 1-separably injective space is a Grothendieck and a
Lindenstrauss\ space.
\end{teor}
\begin{proof}The proof of Proposition \ref{(V)} yields that $1$-separably
injective spaces are of type $\mathcal L_{\infty,1^+}$, that is,
Lindenstrauss\ spaces. It remains to prove that a 1-separably
injective space $E$ must be Grothendieck. It suffices to show that
$c_0$ is not complemented in $E$, so let $\jmath: c_0 \To E$ be an
embedding. Consider an almost-disjoint family $\mathscr M$ of size
$\aleph_1$ formed by infinite subsets of $\N$ and construct the
associated  Johnson-Lindenstrauss twisted sum space
$$\begin{CD}
0 @>>> c_0 @>>> C(\Delta_\mathscr M) @>>> c_0(\aleph_1)@>>> 0.
\end{CD}
$$Observe that the space  $C(\Delta_{\mathscr M})$ has density
character $\aleph_1$, we have therefore a commutative diagram
$$\begin{CD}
0 @>>> c_0 @>>> C(\Delta_\mathscr M) @>>> c_0(\aleph_1)@>>> 0\\
&&@|  @VVV @VVV\\
0 @>>> c_0 @>\jmath >> E @>>> E/\jmath(c_0)@>>> 0.
\end{CD}
$$ If $c_0$ was complemented in $E$ then it would be
complemented in $C(\Delta_\mathscr M)$ as well, which is not.
\end{proof}

Proposition \ref{1-sepinj-CH=univ} leads to the question about the
necessity of the hypothesis {\sf CH}. We will prove now that it
cannot be dropped.

\subsection{A $1$-separably injective but not $1$-universally
separably injective $C(K)$}

\begin{lema}\label{positiveextension}
Let $K,L,M$ be compact spaces and let $f:K\To M$ be a continuous
map, with \mbox{$\jmath=f^{\circ}:C(M)\To C(K)$} its induced
operator, and let $\imath: C(M)\To C(L)$ be a positive norm one
operator. Suppose that $S:C(L)\To C(K)$ is an operator with
$\|S\|=1$ and $S\imath = \jmath$. Then $S$ is a positive operator.
\end{lema}

\begin{proof}
Obviously $S\geq0$ if and only if $S^\ast\delta_x\geq 0$ for all
$x\in K$, where $\delta_x$ is the unit mass at $x$ and
$S^\ast:C(K)^\ast\to C(L)^\ast$ is the adjoint operator. Fix $x\in
K$. By Riesz theorem we have that $S^\ast\delta_x = \mu$ is a
measure of total variation $\|\mu\|\leq 1$. Let $\mu=\mu^+-\mu^-$
be the Hahn-Jordan decomposition of $\mu$, so that
$\|\mu\|=\|\mu^+\|+\|\mu^-\|$, with $\mu^+,\mu^-\geq 0$. We have
that $\delta_{f(x)} = \jmath^\ast\delta_x=\imath^\ast
S^\ast\delta_x=\imath^\ast\mu$, thus
$$
\delta_{f(x)}=\imath^\ast\mu^+-\imath^\ast\mu^-\quad\text{and}\quad
\|\delta_{f(x)}\|=\|\imath^\ast\mu^+\|+\|\imath^\ast\mu^-\|.
$$
Since $\imath$ is a positive operator these imply that the above
is the Hahn-Jordan decomposition of $\delta_{f(x)}$ and so
$\imath^\ast\mu^-=0$, hence $\mu^-=0$.
\end{proof}

\begin{defin}
Let $L$ be a zero-dimensional compact space. An $\aleph_2$-Lusin
family on $L$ is a family $\mathscr{F}$ of pairwise disjoint
nonempty clopen subsets of $L$ with $|\mathscr{F}|=\aleph_2$, such
that whenever $\mathscr{G}$ and $\mathscr{H}$ are subfamilies of
$\mathscr{F}$ with $|\mathscr{G}|=|\mathscr{H}|=\aleph_2$, then
$$\overline{\bigcup\{G\in\mathscr{G}\}} \cap
\overline{\bigcup\{G\in\mathscr{H}\}} \neq\varnothing.$$
\end{defin}

The following lemma shows the consistency of the existence of an
$\aleph_2$-Lusin family on $\mathbb{N}^\ast$. This is rather
folklore of set-theory, but we did not find a reference so we
state it and give a hint of the proof.

\begin{lema} Under {\sf MA} and the assumption $\mathfrak c = \aleph_2$ there exists an
$\aleph_2$-Lusin family on $\mathbb{N}^\ast$.
\end{lema}
\begin{proof} By Stone duality, since the Boolean algebra associated to
$\mathbb{N}^\ast$ is $\wp(\mathbb{N})/\fin$, an
$\aleph_2$-Lusin family on $\mathbb{N}^\ast$ is all the same as an
almost disjoint family $\{A_\alpha\}_{\alpha<\omega_2}$ of
infinite subsets of $\mathbb{N}$ such that for every $B\subset
\mathbb{N}$ either $\{\alpha : |A_\alpha\setminus B|\text{ is
finite}\}$ or $\{\alpha : |A_\alpha\cap B|\text{ is finite}\}$ has
cardinality $<\aleph_2$. Let $\{B_\alpha : \alpha<\omega_2\}$ be
an enumeration of all infinite subsets of $\mathbb{N}$. We
construct the sets $A_\alpha$ inductively on $\alpha$. Suppose
$A_\gamma$ has been constructed for $\gamma<\alpha$. We define a
forcing notion $\mathbb{P}$ whose conditions are pairs $p=(f_p,
F_p)$ where $f_p$ is a $\{0,1\}$-valued function on a finite
subset $\dom(f_p)$ of $\mathbb{N}$ and $F_p$ is a finite subset of
$\alpha$. The order relation is that $p<q$ if $f_p$ extends $f_q$,
$F_p\supset F_q$ and $f_p$ vanishes in $A_\gamma\setminus
\dom(f_q)$ for $\gamma\in F_q$. One checks that this forcing is
ccc. Hence, by {\sf MA}, using a big enough generic filter the forcing
provides an infinite set $A_\alpha\subset \mathbb{N}$ such that,
for all $\gamma<\alpha$,
\begin{enumerate}\item $A_\alpha \cap
A_\gamma$ is finite, and \item If $B_\gamma$ is not contained in
any finite union of $A_\delta$'s, then $A_\alpha\cap B_\gamma$ is
infinite. \end{enumerate} \end{proof}

\begin{teor}\label{consistent}
It is consistent that there exists a compact space $K$ for which the Banach space $C(K)$ is
1-separably injective but not universally 1-separably injective.
\end{teor}

\begin{proof} We will suppose that $\mathfrak c=\aleph_2$ and that there exists
an $\aleph_2$-Lusin family in $\mathbb{N}^\ast$. Under these
hypotheses, let $K$ be the Stone dual compact space of the
Cohen-Parovi\v{c}enko Boolean algebra of \cite[Theorem
5.6]{DowHart}. The definition of that Boolean algebra implies that
$K$ is an $F$-space and thus $C(K)$ is 1-separably injective by
Theorem \ref{people-sep}. We show that it is not universally
1-separably injective. The argument follows the scheme of
\cite[Theorem 5.10]{DowHart}, where they prove that $K$ does not
map onto $\beta\mathbb{N}$, but we use $\aleph_2$-Lusin families
instead of $\omega_2$-chains because they fit better in the
functional analytic context.

Let $\{U_n : n\in\mathbb{N}\}$ be a sequence of pairwise disjoint
clopen subsets of $K$, and let $U=\bigcup_{n}U_n$.
Let $c\subset \ell_\infty$ be the Banach space of convergent
sequences, and $t:c\To C(K)$ be the operator given by $t(z)(x) z_n$ if $x\in U_n$ and  $t(z)(x) = \lim z_n$ if $x\not\in U$.

If $C(K)$ were universally 1-separably injective, we should have
an extension ${T}:\ell_\infty\To C(K)$ of $t$ with
$\|{T}\|=1$. We shall derive a contradiction from the
existence of such operator.

Notice that the
conditions of Lemma~\ref{positiveextension} are applied, so
${T}$ is positive (observe that
$c=C(\mathbb{N}\cup\{\infty\})$ and $T=f^{\circ}$ where
$f:K\To\mathbb{N}\cup\{\infty\}$ is given by $f(x)=n$ if $x\in U_n$ and
$f(x) = \infty$ if $x\notin U$).

For every $A\subset \mathbb{N}$ we will denote $[A] =
\overline{A}^{\beta\mathbb{N}}\setminus\mathbb{N}$. The clopen
subsets of $\mathbb{N}^\ast$ are exactly the sets of the form
$[A]$, and we have that $[A]=[B]$ if and only if $(A\setminus
B)\cup (B\setminus A)$ is finite.

Let $\mathscr F$ be an $\aleph_2$-Lusin family in
$\mathbb{N}^\ast$. For $F=[A]\in \mathscr F$ and
$0<\varepsilon<\frac{1}{2}$, let
$$
F_\varepsilon = \{x\in K\setminus U : {T}(1_A)(x)>1-\varepsilon\}.
$$

Let us remark that $F_\varepsilon$ depends only on $F$ and not on
the choice of $A$. This is because if $[A]=[B]$, then $1_A - 1_B
\in c_0$, hence ${T}(1_A-1_B) = t(1_A-1_B)$ which vanishes out of
$U$, so ${T}(1_A)|_{K\setminus U} = {T}(1_B)|_{K\setminus U}$.

{\sc Claim 1.} If $\delta < \varepsilon$ and $F\in\mathscr F$,
then $\overline{F_{\delta}} \subset F_{\varepsilon}$.
\smallskip

{\sc Claim 2.} $F_\varepsilon\cap G_\varepsilon=\varnothing$ for
every $F\neq G$.
\smallskip

\emph{Proof of Claim 2.} Since $F\cap G=\emptyset$ we can choose
$A,B\subset\mathbb{N}$ such that $F=[A]$, $G=[B]$ and $A\cap
B=\emptyset$. If $x\in F_\varepsilon\cap G_\varepsilon$,
$\bar{T}(1_A+1_B)(x)>2-2\varepsilon>1$ which is a contradiction
because $1_A+1_B = 1_{A\cup B}$ and $\|\bar{T}(1_{A\cup B})\|\leq
\|\bar{T}\|\|1_{A\cup B}\| = 1$. \emph{End of the Proof of Claim
2.}
\smallskip

For every $F\in\mathscr F$, let $\tilde{F}$ be a clopen subset of
$K\setminus U$ such that $\overline{{F}_{0.2}}\subset
\tilde{F}\subset F_{0.3}$. By the preceding claims, this is a
disjoint family of clopen sets. It follows from Proposition 2.6
and Corollary 5.12 in \cite{DowHart} that $K\setminus U$ does not
contain any $\aleph_2$-Lusin family. Therefore we can find
$\mathscr G, \mathscr H\subset \mathscr F$ with $|\mathscr
G|=|\mathscr H|=\aleph_2$ such that
$$
\overline{\bigcup\{\tilde{G} : G\in\mathscr G\}}\cap
\overline{\bigcup\{\tilde{H} : H\in\mathscr H\}}= \varnothing.
$$

Now, for every $n\in\mathbb{N}$ choose a point $p_n\in U_n$. Let
$g:\beta\mathbb{N}\To K$ be a continuous function such that
$g(n)=p_n$.

{\sc Claim 3.} For $u\in\beta\mathbb{N}$, $A\subset \mathbb{N}$,
${T}(1_A)(g(u)) =\left\{%
\begin{array}{ll}
    1, & \text{ if }u\in [A]; \\
    0, & \text{ if }u\not\in [A].\\
\end{array}%
\right.$\\

\emph{Proof of Claim 3.} It is enough to check it for
$u=n\in\mathbb{N}$. This is a consequence of the fact that ${T}$
is positive, because if $m\in A$, $n\not\in A$, then $0\leq
t(1_m)\leq {T}(1_A)\leq t(1_{\mathbb{N}\setminus\{n\}})\leq 1$.
\emph{End of the Proof of Claim 3.}
\smallskip

The function $g$ is one-to-one because
$$
\overline{\{p_n : n\in A\}}\cap \overline{\{p_n : n\not\in A\}}=\varnothing
$$
for every $A\subset\mathbb{N}$, as the function ${T}(1_A)$ separates
these sets. On the other hand, as a consequence of Claim 3 above,
for every $F\in\mathscr F$ and every $\varepsilon$,
$g^{-1}(F_\varepsilon) = F$, and also $g^{-1}(\tilde{F}) = F$.
These facts make the families $\mathscr{H}$ and $\mathscr{G}$
above to contradict that $\mathscr{F}$ is an $\aleph_2$-Lusin family
in $\mathbb{N}^\ast$.
\end{proof}

%\newpage


\begin{thebibliography}{ACCGM}


\bibitem{albiac-kalton-book}
F. Albiac, N.J. Kalton, \emph{Topics in Banach Space Theory}.
Graduate Texts in Mathematics, vol. 233.
   Springer, New York (2006)

\bibitem{A} D.~Amir, \emph{Projections onto continuous function spaces}, Proc.
Amer. Math. Soc. 15 (1964) 396--402.

\bibitem{AL} D. Amir and J. Lindenstrauss.
\emph{The structure of weakly compact sets in Banach Spaces.} Ann.
of Math. 88 (1968), 35-46.


\bibitem{ando}
T. Ando,\emph{ Closed range theorems for convex sets and linear
liftings}, Pacific J. Math. 44 (1973) 393-410



\bibitem{argyros83} S.A. Argyros.
\emph{On the space of bounded measurable functions.} Quart. J.
Math. Oxford Ser. (2) 34 (1983), 129-132.


\bibitem{acgjm} S.A. Argyros, J.M.F. Castillo, A.S. Granero,
M. Jimenez and J.P. Moreno. \emph{Complementation and embeddings
of $c_0(I)$ in Banach spaces.} Proc. London Math. Soc. 85 (2002),
742--772.

\bibitem{bake} J.W. Baker. \emph{Projection constants for $C(S)$
spaces with the separable projection property.} Proc. Amer. Math.
Soc. 41 (1973), 201--204.

%\bibitem{borsuk} K. Borsuk, \emph{\"Uber Isomorphie der
%Funktionalr\"aume}, Bull. Acad. Polon. Sci. Math., (1933), 1--10.

\bibitem{bour} J. Bourgain, \emph{A counterexample to a
complementation problem}, Compo. Math. 43 (1981) 133-144.


\bibitem{BourgainD80} J. Bourgain and F. Delbaen.
\emph{A class of special $\mathcal{L}_\infty$ spaces.} Acta Math.
145 (1980), 155--176.

\bibitem{BourgainLNM} J. Bourgain.
\emph{New classes of $\mathcal{L}^p$-spaces.} Lecture Notes in
Math. 889. Springer-Verlag, 1981.

\bibitem{bp} J. Bourgain and G. Pisier, \emph{A construction of $\mathcal{L}_\infty$-spaces and related Banach spaces}.
Bol. Soc. Bras. Mat. 14 (1983) 109--123.

\bibitem{cabeco} F. Cabello S\'anchez.
\emph{Yet another proof of Sobczyk's theorem.} Methods in Banach
space theory; London Math. Soc. Lecture Notes 337. Cambridge Univ.
Press 2006, pp. 133--138.

\bibitem{cabecastuni} F. Cabello and J.M.F. Castillo.
\emph{Uniform boundedness and twisted sums of Banach spaces.}
Houston J. Math. 30 (2004), 523--536.

\bibitem{cabecastlong} F. Cabello S\'anchez and J.M.F. Castillo.
\emph{The long homology sequence for quasi-Banach spaces, with
applications.} Positivity 8 (2004), 379--394.


\bibitem{ccky}  F. Cabello S\'{a}nchez, J.M.F. Castillo,
N.J. Kalton and D.T. Yost. \emph{Twisted sums with $C(K)$-spaces.}
Trans. Amer. Math. Soc. 355 (2003), 4523--4541.

\bibitem{ccy} F. Cabello S\'anchez, J.M.F. Castillo and D.T. Yost.
\emph{Sobczyk's theorems from A to B.} Extracta Math. 15 (2000)
391--420.

\bibitem{casazza}  P.G. Casazza.
\emph{Approximation properties.} Handbook of the geometry of
Banach spaces, Vol. I; pp. 271--316. North-Holland, Amsterdam,
2001.

\bibitem{castgonz}  J.M.F. Castillo and M. Gonz\'alez.
\emph{Three-space problems in Banach space theory.} Lecture Notes
in Math. 1667. Springer-Verlag, 1997.

\bibitem{castmoresob} J.M.F. Castillo and Y. Moreno.
\emph{Sobczyk's theorem and the Bounded Approximation Property.}
Studia Math.  201 (2010), 1-19.



\bibitem{castmorestud} J.M.F. Castillo, Y. Moreno and J. Su\'arez.
\emph{On Lindenstrauss-Pelczy\'nski spaces.} Studia Math. 174
(2006), 213--231.

\bibitem{castplic} J.M.F. Castillo and A. Plichko,
\emph{Banach spaces in various positions.} J. Funct. Anal. 259
(2010) 2098-2138.

\bibitem{choi-effros}
M.-D. Choi and E.G. Effros, \emph{Lifting problems and the
cohomology of C*-algebras}, Canadian J. Math. 29 (1977) 1092--1111

\bibitem{DashiellLin73} F.K. Dashiell, Jr. and J. Lindenstrauss.
\emph{Some examples concerning strictly convex norms on $C(K)$
spaces.} Israel J. Math. 16 (1973), 329--342.

\bibitem{devilgod}  R. Deville, G. Godefroy and V. Zizler, {\em Smoothness
and renormings in Banach spaces}, vol 64, Pitman Monographs and
Surveys in Pure and Applied Mathematics, 1993.

\bibitem{DowHart} A. Dow and K. P. Hart.
\emph{Applications of another characterization of
$\beta\mathbb{N}\setminus\mathbb{N}$.} Topology Appl. 122 (2002),
105--133.

\bibitem{dowl} P.N. Dowling, \emph{On $\ell_\infty$-subspaces of Banach spaces}, Collect.
Math. 51, 3 (2000),  255-260.

\bibitem{hww}
P. Harmand, D. Werner and W. Werner. \emph{$M$-ideals in Banach
spaces and Banach algebras.} Lecture Notes in Math., 1547.
Springer-Verlag, 1993.

\bibitem{heinrich}
S. Heinrich. \emph{Ultraproducts in Banach space theory.} J. Reine
Angew. Math.  313  (1980), 72--104.

\bibitem{heinrichL1}
S. Heinrich. \emph{Ultraproducts of $L\sb{1}$-predual spaces.}
Fund. Math. 113 (1981), no. 3, 221--234.


\bibitem{HH}
S. Heinrich, C.W. Henson, \emph{Banach space model theory. II.
Isomorphic equivalence}. Math. Nachr. 125 (1986), 301--317.


\bibitem{hensonmoore} C.W. Henson and L.C. Moore.
\emph{Nonstandard hulls of the classical Banach spaces.} Duke
Math. J.  41 (1974), 277--284.

\bibitem{hensonmoore83} C.W. Henson and L.C. Moore.
\emph{Nonstandard analysis and the theory Banach spaces.} In
``Nonstandard analysis--recent developments''. Lecture Notes in
Math. 983; pp. 27--112. Springer-Verlag, 1983.

\bibitem{hiltstam}
E. Hilton and K. Stammbach. \emph{A course in homological
algebra.} Graduate Texts in Math. 4. Springer-Verlag, 1970.

\bibitem{johnlind} W.B. Johnson and J. Lindenstrauss.
\emph{Some remarks on weakly compactly generated Banach spaces.}
Israel J. Math. 17 (1974), 219--230.


\bibitem{johnoikh} W.B. Johnson and T. Oikhberg.
\emph{Separable lifting property and extensions of local
reflexivity.} Illinois J. Math. 45 (2001), 123--137.


\bibitem{johnzippre} W. B. Johnson and M. Zippin, {\it Separable $L_1$ preduals are quotients of $C(\Delta)$}, Israel
J. Math. 16 (1973) 198-202.

\bibitem{koszpams}  P. Koszmider.
\emph{On decomposition of Banach spaces of continuous functions on
Mr\'owka's spaces}, Proc. Amer. Math. Soc. 133 (2005), 2137-2146.


\bibitem{kuramos}  K. Kuratowski and A. Mostowski.
\emph{Set Theory.} P.W.N.; Warszawa, 1968.


\bibitem{lind} J. Lindenstrauss.
\emph{On the extension of compact operators.} Memoirs of the Amer.
Math. Soc. 48, 1964.

\bibitem{lindpams} J. Lindenstrauss.
\emph{On the extension of operators with range in a $C(K)$ space.}
Proc. Amer. Math. Soc. 15 (1964), 218--225.

\bibitem{lindtzaf}  J. Lindenstrauss and L. Tzafriri,
\emph{Classical Banach spaces I.} Springer-Verlag, 1977.

\bibitem{marc}  W. Marciszewski.
\emph{On Banach spaces $C(K)$ isomorphic to $c_0(\Gamma)$.} Studia
Math. 156 (2003), 295-302

\bibitem{marcpol}  W. Marciszewski and R. Pol,
\emph{On Banach spaces whose norm-open sets are $F_\sigma$ sets in
the weak topology}, J. Math. Anal. Appl. 350 (2009), 708-722.


\bibitem{ostr} M. I. Ostrovskii.
\emph{Separably injective Banach spaces.} Functional Anal. i.
Prilozhen 20 (1986), 80-81. English transl.: Functional Anal.
Appl. 20 (1986), 154-155.

\bibitem{part}
J.R. Partington. \emph{Subspaces of certain Banach sequence
spaces}, Bull London Math. Soc. 13 (1981), 162-166.


\bibitem{Pe}
A.~Pe{\l}czy\'nski, \emph{Projections in certain Banach spaces},
Studia Math. 19 (1960) 209--228.

\bibitem{disjoint} H.P. Rosenthal.
\emph{On relatively disjoint families of measures, with some
applications to Banach space theory.} Studia Math. 37 (1970),
13--36.


\bibitem{roseco} H.P. Rosenthal.
\emph{The complete separable extension property.} J. Operator
Theory 43 (2000,) 329--374.

\bibitem{sims} B. Sims.
\emph{``Ultra''-techniques in Banach space theory.} Queen's Papers
in Pure and Applied Mathematics, 60. Kingston, ON, 1982.

%\bibitem{sob1} A. Sobczyk.
%\emph{On the extension of linear transformations.} Trans. Amer.
%Math. Soc. 55 (1944) 153--169.

%\bibitem{sob2} A. Sobczyk.
%\emph{Projection of the space m on its subspace $c_0$.} Bull.
%Amer. Math. Soc. 47 (1941), 938--947.

\bibitem{stern} J. Stern.
\emph{Ultrapowers and local properties of Banach spaces.} Trans.
Amer. Math. Soc. 240 (1978), 231--252.

%\bibitem{zipp} M. Zippin.
%\emph{The separable extension problem.} Israel J. Math. 26 (1977),
%372--387.

\end{thebibliography}
\end{document}

%\noindent{\footnotesize{Version: May 2016}}

%\maketitle

\author{A. Avil\'es, F. Cabello S\'anchez, J. M. F. Castillo, M. Gonz\'alez, Y. Moreno.}

\thanks{The research of the first named author has been supported in part by projects
MTM2014-54182-P funded by MINECO and FEDER 19275/PI/14 (PCTIRM 2011-2014) funded by Fundaci\'on S\'{e}neca - Agencia de Ciencia y Tecnolog\'\i a de la Regi\'on de Murcia.
The research of the other four authors has been supported in part by project MTM2016-76958-C2; and that of authors 2,3 and 5 by Ayuda a Grupos GR15152 de la Junta de Extremadura
}

%\begin{abstract}
%\end{abstract}
\begin{abstract}
We show that, under the continuum hypothesis, ``to be universally separably injective" is not a $3$-space property, as we wrongly claimed in the paper mentioned in the title.
\end{abstract}

\maketitle

\section{Introduction}
A Banach space $E$ is said to be universally separably injective when every operator from a separable Banach space into $E$ can be extended
to any superspace. This notion was introduced and studied in \cite{accgm1}, where it was claimed that it is a $3$-space property. During the preparation of the monograph \cite{2132} we detected a serious gap in the proof of \cite[Proposition 5.3]{accgm1}; in fact, Section 6.2 of \cite{2132} was then devoted to clarify the situation: which results were still true, which equivalences hold and which results were affected and remained unproved until the $3$-space problem could be solved. Reduced to its bare bones the situation was this: the $3$-space property for universal separable injectivity is equivalent to the assertion that $\Ext(\ell_\infty, U)=0$ for every universally separably injective space $U$ and implies that $\Ext(\ell_\infty/c_0, \ell_\infty/c_0)=0$.\medskip

This note contains a construction which, under the continuum hypothesis ({\bf{\sf CH}}), becomes a counter-example for the 3-space problem for universal separable injectivity. We will show that, under  {\bf{\sf CH}}, one has
$$\Ext(\ell_\infty/c_0, \ell_\infty/c_0)\neq 0$$and moreover, there is an exact sequence $0\to  \ell_\infty/c_0  \to X\to \ell_\infty\to 0$ in which $X$ is not universally separably injective, which we show providing an explicit operator $c_0\to X$ that cannot be extended. The latter example is pertinent since the proof of the corresponding implication in \cite{accgm1} was not constructive.

\section{The counter-example}
Here we will adhere to the notations of \cite{accgm1}, in particular $\Ext(A,B)=0$ means that every short exact sequence $0\to B\to X\to A\to 0$ splits.\medskip% The reader is referred to \cite{accgm1} for all unexplained notation

In order to construct the required counter-example we need an extension of the form
$$
\begin{CD}
0 @>>> C(K) @>>> X @>>> C(L)@>>> 0
\end{CD}
$$
($K$ and $L$ metric compacta) which is not trivial and remains nontrivial after applying ultrapowers. Those acquainted with the ``model theory'' of Banach spaces (see \cite{h-i} for an introduction) will realize that to prove that a given exact sequence has nontrivial ultrapowers one has to prove that the starting sequence does not split by means of an argument belonging to ``the first order logic'' of Banach spaces. This is a bit inacccurate: the reader should not be worried since all computations will be carried out explicit.

\subsection*{The cornerstone example...}

We will use a slight variation of the Foia\c{s}-Singer twisted sum of $C[0,1]$ and $c_0$ in \cite[Theorem~6]{fs}. Let $\Delta=\{0,1\}^\N$ be the Cantor set which is equipped with the product topology and the ``lexicographic'' order.
We denote by $\Delta_0$ the (countable, dense) subset of those $t\in\Delta$ having finitely many 1's. Let $D$ be the space of all functions $\Delta\to\R$ that are continuous at every $t\notin\Delta_0$ and left continuous with right limits at every $t\in\Delta_0$. It is \emph{really} easy to prove that the sup norm makes $D$ into a Banach space containing $C(\Delta)$ and that the quotient $D/C(\Delta)$ is isometric to $c_0$. More precisely, if $J:D\to\ell_\infty(\Delta_0)$ denotes the ``jump'' function $J(f)(q)={1 \over 2}(f(q^+)-f(q))$, then $J$ maps $D$ onto $c_0(\Delta_0)$ and $\dist(f,C(\Delta))=\|Jf\|_\infty$ for every $f\in D$. Thus one has an extension
\begin{equation}\label{D}
\begin{CD}
0 @>>> C(\Delta) @>\imath >> D@>J>> c_0@>>> 0
\end{CD}
\end{equation}

There are many reasons why the preceding sequence is not trivial (cf. \cite[Remark (ii)]{a-l}, \cite[Example 1.20, p. 24]{b-l}, \cite[Lemma 2.2]{tams}). For our present purposes we need to add a new one:

\begin{fact}
Let $(f_q)$ be any family (sequence) in $D$ such that $J(f_q)=e_q$ for every $q\in\Delta_0$.
Then, given $\lambda^1,\dots,\lambda^k\in\R; q^1,\dots,q^n\in \Delta_0$ and $\e>0$, there exist $q\in\Delta_0\backslash\{q^1,\dots,q^n\}$ and $\lambda=\pm1$ such that
$$
\left\|\lambda f_q+ \sum_{k=1}^n \lambda^k f_{q^k}\right\|_D \geq  1 +  \left\| \sum_{k=1}^n \lambda^k f_{q^k}\right\|_D-\e.
$$
\end{fact}

\begin{proof}
With no serious loss of generality we may assume that there is  $q$ in $\Delta_0\backslash\{q^1,\dots,q^n\}$ such that
$$
\sum_{k=1}^n \lambda^k f_{q^k}(q)> \left\| \sum_{k=1}^n \lambda^k f_{q^k}\right\|_D-\e.
$$
But  $\sum_{k=1}^n \lambda^k f_{q^k}$ is continuous at $q$ and so $\sum_{k=1}^n \lambda^k f_{q^k}(q^+)=\sum_{k=1}^n \lambda^k f_{q^k}(q)$.  Now, since $Jf_q=e_q$, if $f_q(q)\geq -1$, then $f_q(q^+)\geq 1$, hence
$$
\left\|f_q+ \sum_{k=1}^n \lambda^k f_{q^k}\right\|_D\geq  \left(f_q+ \sum_{k=1}^n \lambda^k f_{q^k}\right)(q^+)> 1+\left\| \sum_{k=1}^n \lambda^k f_{q^k}\right\|_D-\e,
$$
as required. And if $f_q(q)< -1$, then
$$
\left\|-f_q+ \sum_{k=1}^n \lambda^k f_{q^k}\right\|_D\geq \left(-f_q(q)+ \sum_{k=1}^n \lambda^k f_{q^k}(q)\right)> 1 +\left\| \sum_{k=1}^n \lambda^k f_{q^k}\right\|_D-\e.\vspace{-20pt}
$$

\end{proof}

The immediate consequence is that every linear section $L:c_0(\Delta_0)\To D$ of $J$ is unbounded in $c_{00}(\Delta_0)$. Indeed if $J\circ L$ is the identity on $c_0(\Delta_0)$, then letting $f_q=L(e_q)$ and taking $\e={1\over 2}$ in Fact~1 one obtains a sequence $(q^k)$ and signs $\lambda^k=\pm1$ such that
$$
\left\|L\left( \sum_{k=1}^n \lambda^k e_{q^k}\right)\right\|=\left\| \sum_{k=1}^n \lambda^k f_{q^k}\right\|_D\geq \frac{n}{2},
$$
while $\| \sum_{k=1}^n \lambda^k e_{q^k}\|_{c_0(\Delta_0)}= 1$ for all $n\in\mathbb N$.
\medskip

The original example in \cite{fs} uses the space of bounded functions $[0,1]\to\mathbb R$ which are continuous except on the ``dyadic'' rational numbers, where they are left continuous and have a right limit. We have worked with $\Delta$ instead of $[0,1]$ because, under the continuum hypothesis, countable ultrapowers of $C(\Delta)$ are isometric to $\ell_\infty/c_0$ (by Parovi\u{c}enko's theorem; see Bankston \cite[Proposition 2.4.1]{b} or \cite[Proposition 4.12]{2132}) which is the space that we need to twist, while countable ultrapowers of $C[0,1]$ are merely isomorphic to $\ell_\infty/c_0$.

\subsection*{...and its ultrapowers}

Let us consider ultrapowers of the extension (\ref{D}). Let $I$ be an ``index'' set and $\U$ an ultrafilter on $I$. Take the ultraproduct sequence
\begin{equation}\label{DU}
\begin{CD}
0 @>>> C(\Delta)_\U @>>> D_\U@>J_\U>> c_0(\Delta_0)_\U@>>> 0.
\end{CD}
\end{equation}
Observe that if $(q_i)$ is a family of points of $\Delta_0$ indexed by $I$, then the class of $(e_{q_i})$ in the ultrapower $(c_0(\Delta_0))_\U$ depends only on the class of $(q_i)$ in the set-theoretic ultrapower $\Delta_0^\U$. Thus, given $q\in \Delta_0^\U$ let us write $e_q=[(e_{q_i})]$, where $\langle(q_i)\rangle=q$. Clearly, if $q^1,\dots,q^n$ are different points of  $\Delta_0^\U$, then
$$
\left\|\sum_{k=1}^n \lambda_k e_{q^k}\right\|_{(c_0(\Delta_0))_\U}= \max_{1\leq k\leq n}|\lambda_k|,
$$
which provides an isometric embedding of $c_0(\Delta_0^\U)$ into $(c_0(\Delta_0))_\U$.

\begin{fact} Let $(f_q)$ be any family in $D_\U$ such that $J_\U(f_q)=e_q$ for every $q\in\Delta_0^\U$.
Then, given $\lambda^1,\dots,\lambda^n\in\R; q^1,\dots,q^n\in \Delta_0^\U$ and $\e>0$, there exist $q\in\Delta_0^\U\backslash\{q^1,\dots,q^n\}$ and $\lambda=\pm1$ such that
$$
\left\|\lambda f_q+ \sum_{k=1}^n \lambda^k f_{q^k}\right\|_{D_\U} \geq 1 +  \left\| \sum_{k=1}^n \lambda^k f_{q^k}\right\|_{D_\U}-\e.
$$
\end{fact}

\begin{proof}
The proof is the same as before. Indeed, the preceding proof consists of ``sentences'' which cannot ``distinguish'' either a given exact sequence from its ultrapowers or $\Delta_0$ from $\Delta_0^\U$. Just for fun, let us display the ``interpretation'' of the above formulae in the ultrapower structure. By the way, any attempt to prove Fact 2 by applying Fact 1 coordinatewise is doomed to fail.\medskip

Notice that if $f\in D_\U$ and $t\in\Delta^\U$, then the ``value of $f$ at $t$'' is given by
$$
f(t)=\lim_{\U(i)} f_i(t_i),
$$
where $(f_i)$ is any representative of $f$ and $(t_i)$ is any representative of $t$. Clearly,
$$
\|f\|_{D_\U}= \sup_{q\in \Delta_0^\U}|f(q)|.
$$
Also, note that for each $q\in\Delta_0^\U$ we can define the ``right-limit''
$$f(q^+)=\lim_{\U(i)} f_i(q_i^+),
$$
where $(f_i)$ is any representative of $f$ and $(q_i)$ is any representative of $q$. Now, let us copy and paste.\medskip

We may assume that there is $q\in\Delta_0^\U\backslash\{q^1,\dots,q^n\}$ such that
$$
\sum_{k=1}^n \lambda^k f_{q^k}(q)> \left\| \sum_{k=1}^n \lambda^k f_{q^k}\right\|_D-\e.
$$
Clearly $f_{q^k}(q^+)=f_{q^k}(q)$ for $1\leq k\leq n$. Now, since $J_\U(f_q)=e_q$, if $f_q(q)\geq -1$, then $f_q(q^+)\geq 1$ and
$$
\left\|f_q+ \sum_{k=1}^n \lambda^k f_{q^k}\right\|_{D_\U}\geq \left(f_q+ \sum_{k=1}^n \lambda^k f_{q^k}\right)(q^+)\geq 1+
\left\|\sum_{k=1}^n \lambda^k f_{q^k}\right\|_{D_\U}-\e.
$$
 And if $f_q(q)< -1$, then
$$
\left\|-f_q+ \sum_{k=1}^n \lambda^k f_{q^k}\right\|_{D_\U}\geq \left(-f_q(q)+ \sum_{k=1}^n \lambda^k f_{q^k}(q)\right)> 1 +\left\| \sum_{k=1}^n \lambda^k f_{q^k}\right\|_{D_\U}-\e,
$$
as required.
\end{proof}

\section{Consequences}
The following result stablishes a number of consequences of Fact~2. Some of them (namely (a) and (b)) are ``absolute'' (that is, they hold in {\bf{\sf ZFC}}), while the proofs of (c), (d), (e) and (f) hold under {\bf{\sf{CH}}}. We believe that they are still true in {\bf{\sf ZFC}}, nevertheless.

\begin{teor}Let $\U$ be an ultrafilter on an index set $I$. Then:
\begin{itemize}
\item[(a)] $\Ext(C(\Delta)_\U, C(\Delta)_\U)\neq 0$.
\item[(b)] $\Ext(c_0(\N^\U), C(\Delta)_\U)\neq 0$.
\end{itemize}
Moreover, under {\sf CH}, one has:
\begin{itemize}
\item[(c)] $\Ext(c_0(\aleph_1), \ell_\infty/c_0)\neq 0$.
\item[(d)] $\Ext(\ell_\infty/c_0, \ell_\infty/c_0)\neq 0$.
\item[(e)] $\Ext(\ell_\infty, \ell_\infty/c_0)\neq 0$.
\item[(f)] ``To be universally separably injective'' is not a three-space property.
\end{itemize}
\end{teor}

\begin{proof} (a) The exact sequence (\ref{DU}) is not trivial, so $\Ext((c_0)_\U, C(\Delta)_\U)\neq 0$. Since $c_0$ is a direct factor of $C(\Delta)$ (actually $c_0$ is a direct factor in every separable Lindenstrauss space), it follows that $C(\Delta)_\U$ contains a complemented copy of $(c_0)_\U$ and the result follows.

(b) As we mentioned before, the space $c_0(\Delta_0^\U)$, which is isometric to $c_0(\N^\U)$, can be seen as a closed subspace of $(c_0(\Delta_0))_\U$. From (\ref{DU}) one can therefore form the exact sequence
$$%\begin{equation}\label{DU}
\begin{CD}
0 @>>> C(\Delta)_\U @>>> J_\U^{-1}[c_0(\Delta_0^\U)]@>J_\U>>c_0(\Delta_0^\U) @>>> 0,
\end{CD}
$$%\end{equation}
and Fact~2 obviously implies that this extension is not trivial.
\medskip

Before going into the proof of the other statements, let us explain how {\bf{\sf CH}} (the statement $2^{\aleph_0}=\aleph_1$) enters in our arguments.
As is well-known, the quotient algebra $\ell_\infty/c_0$ is isometrically isomorphic (in {\bf{\sf ZFC}}) to $C(\mathbb N^*)$. Here, $\mathbb N^*=\beta\mathbb N\backslash\mathbb N$ is the growth of the integers in the Stone-\v Cech compactification. As we mentioned before, the continuum hypothesis implies that if $\U$ is a nontrivial ultrafilter on the integers, then the algebra $C(\Delta)_\U$ is isometric to $C(\N^*)$, hence to $\ell_\infty/c_0$.\medskip

(c) The continuum hypothesis implies that if $\U$ is a nontrivial ultrafilter on the integers, then $C(\Delta)_\U$ is isometric to $C(\N^*)$ and that the cardinality of $\N^\U$ (which is always the continuum) is $\aleph_1$.

(d) follows from (a).

(e) Follows from (d) as in the proof of ``(3) implies (4)'' in \cite[Proposition~6.1]{2132}.
A simpler proof is now at hand; we will use the same notation as in that paper. Pick any nontrivial exact sequence
$$\begin{CD}
0 @>>> \ell_\infty/c_0 @>{\imath}>> X @>{\pi}>> \ell_\infty/c_0 @>>> 0
\end{CD}
$$
and let $q: \ell_\infty\to \ell_\infty/c_0$ be the natural quotient map. Look at the diagram
$$\begin{CD}
0 @>>> \ell_\infty/c_0 @>{\imath}>> X @>{\pi}>> \ell_\infty/c_0 @>>> 0\\
&&&&&&@AAqA\\
&&&&&& \ell_\infty
\end{CD}
$$
Let us see that $q$ cannot be lifted to $X$. Assume on the contrary that there is $Q:\ell_\infty\to X$ such that $\pi\circ Q=q$. Then the restriction of $Q$ to $c_0$ takes values in $\imath[\ell_\infty/c_0]=\ker\pi$ and since $\imath[\ell_\infty/c_0]$ is universally separably injective one can find an operator $Q_0:\ell_\infty\to X$ which takes values in $\imath[\ell_\infty/c_0]$ and extends the restriction of $Q$ to $c_0$. Clearly $Q-Q_0: \ell_\infty\to X$ factors through an operator $\ell_\infty/c_0\to X$ which provides a section of $\pi$.\medskip

Now, look at the commutative so-called pull-back diagram in which $\PB= \{(x,z)\in X \oplus_\infty \ell_\infty: \pi(x)=q(z)\}$ with $'\imath(f)=(\imath(f),0)$, $'q(x,z)=z$ and $'\pi(x,z)=x$:
\begin{equation}\label{pbdiagram}\begin{CD}
0 @>>>  \ell_\infty/c_0 @>{\imath}>> X @>{\pi}>> \ell_\infty/c_0 @>>> 0\\
&&@| @AA'qA @AAqA\\
0 @>>>  \ell_\infty/c_0 @>'\imath>> \PB @>'\pi>> \ell_\infty @>>> 0
\end{CD}
\end{equation}
The standard homological fact (see \cite[Lemma 2.2]{accgm1} or \cite[Lemma A.21]{2132}) that the lower sequence in a pull-back diagram as above splits if and only if the right map
(in this case $q$) can be lifted means that the lower sequence does not split since $q$ cannot be lifted through $\pi$.\medskip

Part (f) follows from (d) (or from (e)) as in the proof of ``(1) implies (3)'' in \cite[Proposition~6.1]{2132}.
However, we can use the peculiarities of diagram (\ref{pbdiagram}) and the sequence (\ref{DU}) to exhibit a counter-example which is ``explicit", modulo {\sf CH}. Recall that
we know from (a) and (b) that, under {\sf CH} there is an exact sequence
$$\begin{CD}
0 @>>> Y @>\imath>> X@>\pi>> Z @>>> 0
\end{CD}
$$
in which $Y$ is universally separably injective (the space $C(\Delta)_\U$), $Z$ is universally separably injective as well (the space $\ell_\infty/c_0$) and $Z$ contains
a subspace $E$ isomorphic to $c_0(\aleph_1)$ (the space $ c_0(\Delta_0^\U)$) where the quotient map is not invertible. Let us call $u: E\to Z$ the canonical inclusion. For the sake of clarity we will write $\ell_\infty/c_0$ instead of $Z$ from now on. Let us form the commutative pull-back diagram
$$\begin{CD}
0 @>>> Y @>\imath >> X@>\pi>>\ell_\infty/c_0 @>>> 0\\
&&@| @AA'qA @AAqA\\
0 @>>>  Y @>'\imath >> \PB @>'\pi>> \ell_\infty @>>> 0
\end{CD}
$$
We claim that the space $\PB$ is not universally separably injective. To prove it, we complete the diagram adding the kernels of $q$ and $'q$ to get (the reader should ignore the dotted arrows at this moment)

$$\xymatrix{
&&&& 0 && 0\\
0 \ar[rr] && Y \ar[rr]_\imath\ar@{=}[d]  && X \ar@{.>}@/^1pc/[d]^{'\sigma}\ar[rr]_\pi \ar[u] && \ell_\infty/c_0  \ar@{.>}@/^1pc/[d]^{\sigma} \ar[rr] \ar@{.>}@/_1pc/[ll]_{s} \ar[u] && 0\\
0 \ar[rr] && Y   \ar[rr]_{'\imath}  && \PB \ar[rr]_{'\pi} \ar[u]^{'q} && \ell_\infty \ar[rr]\ar[u]^q  \ar@{.>}@/_1pc/[ll]_{s} && 0\\
&&&&c_0 \ar[u]^{'\jmath}\ar@{=}[rr] && c_0 \ar[u]^{\jmath}\\
&&&& 0\ar[u] && 0\ar[u]
}
$$
Let us show that the inclusion $'\jmath: c_0 \to \PB$ cannot be extended to $\ell_\infty$ through $\jmath$. If an operator $'s: \ell_\infty\to \PB$ exists such that $'s\circ\jmath={}'\jmath$ then it induces an operator
$s: \ell_\infty/c_0\To X$ so that $s\circ q= {}'q\circ{}'s$. The so-called ``parallel lines principle'' (\cite[Proposition 3.2]{castplic}) provides  operators $'\sigma: X \To \PB$ and $\sigma: \ell_\infty/c_0\to \ell_\infty$ so that $\sigma\circ\pi={}'\pi\circ{}'\sigma$ and $'\imath={}'\sigma\circ\imath$, that is to say,
the following is a commutative diagram too:
$$\xymatrix{
0 \ar[rr] && Y \ar[rr]\ar@{=}[d]  && \PB \ar[rr]^{'\pi}  && \ell_\infty \ar[rr]  && 0\\
0 \ar[rr] && Y  \ar[rr]  && X \ar[rr]^{\pi} \ar[u]^{'\sigma} && \ell_\infty/c_0 \ar[rr]\ar[u]^\sigma  && 0
}
$$
If we ``pull-back again'' with the inclusion $u$ then  we obtain the commutative diagram
$$\xymatrix{
0 \ar[rr] && Y \ar[rr]\ar@{=}[d]  && \PB \ar[rr]^{'\pi}  && \ell_\infty \ar[rr]  && 0\\
0 \ar[rr] && Y \ar[rr]\ar@{=}[d]  && X \ar[rr]^{\pi} \ar[u]^{'\sigma} && \ell_\infty/c_0 \ar[rr]\ar[u]^\sigma  && 0\\
0 \ar[rr] && Y \ar[rr]  && \pi^{-1}[E]\ar[rr] \ar[u]^{ 'u} && E \ar[rr]\ar[u]^{u} && 0
}
$$
in which the lower sequence does not split. But since every operator $c_0(\Gamma)\To \ell_\infty$ has separable range for no matter which set $\Gamma$, the operator $\sigma u$ factorizes through a separable space; and since $Y$ is separably injective, the lower exact sequence
should split, something it does not.\end{proof}

This shows that most statements in \cite[Section 5]{accgm1} (namely Propositions 5.3, 5.4 and 5.6 and  Theorem 5.5) are false, at least if one assumes the continuum hypothesis.
The mistake in the proof of Proposition 5.3 occurs as follows: It is shown first that universal separable injectivity is equivalent to $\ell_\infty$-uppersaturation (every separable subspace is contained in a copy of $\ell_\infty$ contained in the space, see \cite{accgm1,2132}); then, the purpose is to show that $\ell_\infty$-uppersaturation is a $3$-space property.\medskip

The argumentation in \cite{accgm1} is: let $0\to Y \to X \stackrel{q}\to Z\to 0$ be an exact sequence in which both $Y,Z$ are $\ell_\infty$-uppersaturated and let $S$ be a separable subspace of $X$. The closure $[q(S)]$ of $q(S)$ in $Z$ is separable and thus a standard result going back to \cite{valdv} yields separable subspaces $S_0\subset Y$ and $S_1\subset X$ with $S\subset S_1$ so that $S_1/S_0= [q(S)]$. Now, $S_0$ is contained in a copy $Y_\infty$ of $\ell_\infty$ inside $Y$. So, $S_1$ is contained in the subspace $Y_\infty\oplus [q(S)]$  of $X$. Since $[q(S)]$ is contained in a copy $Z_\infty$ of $\ell_\infty$ in $Z$, so it is claimed in \cite{accgm1}, $Y_\infty\oplus [q(S)]$ is contained in a copy of $Y_\infty\oplus Z_\infty$ inside $X$; which is easy to believe, but false.\medskip

A comparison with the (right) proof that $c_0$-uppersaturation is a 3-space property (that can be seen in detail in \cite[Proposition 6.2]{2132}) is clarifying: replace now $Y_\infty, Z_\infty$ by copies $Y_0, Z_0$ of $c_0$ to get at the last step that $Y_0\oplus [q(S)]$ is contained in a copy of $Y_0\oplus Z_0$ inside $X$; which is true! The ``invisible" reason behind such different behaviour is that while in the $c_0$ case there is a commutative diagram

$$\begin{CD}
0 @>>> Y @>>> X@>>>Z @>>> 0\\
&&@AAA @AAA @AAA\\
0 @>>> Y_0 @>>> X_0@>>>c_0 @>>> 0\\
&&&&@AAA @AAA\\
&&&&S @>>> [q(S)]\\
\end{CD}$$
in which $Y_0$ can be chosen separable, and thus the canonical inclusion $Y_0 \to Y$  factorizes through $c_0$, in the $\ell_\infty$ case
$$\begin{CD}
0 @>>> Y @>>> X@>>>Z @>>> 0\\
&&@AAA @AAA @AAA\\
0 @>>> Y_0 @>>> X_0@>>>\ell_\infty @>>> 0\\
&&&&@AAA @AAA\\
&&&&S @>>> [q(S)]
\end{CD}$$
there is no obvious choice of $Y_0$ such that the canonical inclusion $Y_0 \to Y$  factorizes through $\ell_\infty$. And in the end such choice is impossible. In particular, it turns out in the counterexample above that the copy $'j(c_0)$ of $\PB$ is not contained into any copy of $\ell_\infty$ contained in $\PB$ as it was wrongly claimed during the  proof of \cite[Proposition 5.3]{accgm1}. Theorem 5.2 and Corollary 5.7 are true (in {\sf ZFC}) and their proofs are correct. All other results in \cite{accgm1} are independent of the wrong ones. The assertion in Example 4.5 (b) in \cite{accgm3} is also unjustified as we explained in \cite[p. 175]{2132}. Assuming {\sf CH}, the result in (a) yields a nontrivial exact sequence
$$\begin{CD}
0 @>>> C(\N^*) @>>> C(\N^*)@>>>C(\N^*) @>>> 0\end{CD}$$
which provides an alternative proof and improvement for \cite[Corollary 2.45]{2132}: $C(\N^*)$ contains an uncomplemented copy of itself.
\medskip

\subsection{Final remarks}  We could not decide whether universal separable injectivity is a $3$-space property without assuming {\sf CH}. Observe, however, that {\sf CH} appears only to guarantee that for some (in fact, for every) free ultrafilter $\mathscr U$ on $\N$ the ultrapower $C(\Delta)_{\mathscr U}$ is isomorphic (actually, isometric) to  the Banach space $\ell_\infty/c_0=C(\N^*)$ (or isomorphic to a quotient of $\ell_\infty$ by a separable subspace, for what matters).
According to \cite[Theorem 2.12]{hhy} plus \cite[Corollary 3.7]{henson} (see also  \cite{moore}), there is an ultrafilter $\V$ and a function $k: \N\to \N$ so that if $c$ denotes the space of convergent sequences then $c_\V$ is isometric to the ultraproduct $({\ell_\infty^{k(n)})}_\V$, and this space can be written as a quotient of $\ell_\infty$ ... but the quotient map has nonseparable kernel.  Assuming that {\sf CH} fails, Farah and Shelah proved in \cite[Proposition 8.2]{f-s} that there are at least $2^{\frak c}$ ultrafilters on the integers for which the corresponding ultrapowers of $C(\Delta)$ are mutually nonisomorphic as Banach algebras and, therefore, the underlying Banach spaces are not isometric. We do not know, however, if one can find an ultrafilter $\mathscr U$ so that $C(\Delta)_{\mathscr U}$ is still isomorphic to  $\ell_\infty/c_0$, as a Banach space, which is what David Sherman called in \cite{sherman} to be ``a very bad question". In any case one cannot prove \cite[Lemma 4.10]{b} that $\ell_\infty/c_0=C(\N^*)$ is isometric to an ultraproduct of any family of $C(K)$-spacesin {\sf ZFC}: as it is remarked in \cite[p. 119]{2132} the ideal maximal space of the latter algebra always has $P$-points, a fact that cannot be proved for $\N^*$ in {\sf ZFC} by a result of Shelah \cite{w}.

\small{\address{Departamento de Matem\'aticas, Universidad de Murcia,
30100 Espinardo, Murcia, Spain} \email{avileslo@um.es}

\address{Departamento de Matem\'aticas, Universidad de Extremadura, Avenida de Elvas s/n, 06071 Badajoz, Spain}
             \email{fcabello@unex.es}

\address{Departamento de Matem\'aticas, Universidad de Extremadura, Avenida de Elvas s/n, 06071 Badajoz, Spain}
             \email{castillo@unex.es}

\address{Departamento de Matem\'aticas, Universidad de Cantabria, Avenida los Castros s/n, 39071 Santander, Spain}
             \email{manuel.gonzalez@unican.es}

\address{Escuela Polit\'ecnica, Universidad de Extremadura, Avenida de la Universidad s/n, 10071 C\'aceres, Spain}
             \email{ymoreno@unex.es}}